\documentclass[reqno,12pt]{amsart}
\usepackage[left=3.5cm, top=3cm, bottom=3cm, right=3.5cm]{geometry}
\usepackage{amssymb}
\usepackage{amsxtra}
\usepackage{amsmath}
\usepackage{amsfonts}

\numberwithin{equation}{section}

\newtheorem{thm}{Theorem}[section]
\newtheorem{Mthm}{Main Theorem}
\newtheorem{prop}[thm]{Proposition}
\newtheorem{cor}[thm]{Corollary}

\newtheorem{lem}[thm]{Lemma}
\newtheorem{de}[thm]{Definition}

\newtheorem{ex}[thm]{Example}

\def\ty{{\widetilde y}}
\def\w{\omega}
\def\bx{\bold x}
\def\av{\alpha^{\vee}}
\def\tw{\widetilde W}

\newcommand{\eqa}{\begin{eqnarray}}
\newcommand{\eeqa}{\end{eqnarray}}
\newcommand{\beq}{\begin{equation}}
\newcommand{\eeq}{\end{equation}}
\newcommand{\nn}{\nonumber}
\newcommand{\p}{\partial}

\newcommand{\ve}{\epsilon}

\newcommand{\lm}{\lambda}

\newcommand{\pal}{\partial}

\newcommand{\al}{\alpha}

\newcommand{\sg}{\sigma}

\def \ep{\epsilon}

 \def \dsum{\displaystyle\sum}
 \def\res{\mathop{\text{\rm res}}}

\begin{document}

\title[]
{Frobenius manifolds and a new class of  extended affine Weyl groups $\widetilde{W}^{(k,k+1)}(A_l)$ }
\author[]{Dafeng Zuo}


\address{School of Mathematical Science,
University of Science and Technology of China,
 Hefei 230026,
P.R.China}

\email{dfzuo@ustc.edu.cn}

\subjclass[2000]{Primary 53D45; Secondary 32M10}

\keywords{Extended affine Weyl group, Frobenius manifold}

\date{\today}

\begin{abstract} We present a new class of extended affine Weyl groups $\widetilde{W}^{(k,k+1)}(A_l)$ for
$1\leq k <l$ and obtain an analogue of Chevalley-type theorem for their invariants. We further show the existence
of Frobenius manifold structures  on the orbit spaces of $\widetilde{W}^{(k,k+1)}(A_l)$ and also construct
Landau--Ginzburg superpotentials for these Frobenius manifold structures.

\end{abstract}

\maketitle 
{\tableofcontents}
\newpage
\section{Introduction}

 E.Witten, R.Dijkgraaf, E.Verlinde
and H.Verlinde (\cite{W1990,DVV1991}) in 1990's introduced a remarkable system of
partial differential equations, i.e., WDVV equations of associativity,  on
two dimensional topological field theory (briefly 2dTFT). In order to understand a geometrical
foundation of 2dTFT on the bases of WDVV equations, B.Dubrovin (\cite{Du1992,Du1996}) extended
the Atiyah's axioms of 2dTFT (\cite{A1988}) and  invented
a nice geometrical object, that is, Frobenius manifold designed as a coordinate-free
formulation of WDVV equations.

 For an arbitrary $n$-dimensional Frobenius manifold, B.Dubrovin (\cite{Du1996}) defined a monodromy group, which acts
on $n$-dimensional  linear space and can be regarded as (an extension of) a group generated by reflections.
The Frobenius manifold itself can be identified with the orbit space of the group in the sense to be
specified for each class of monodromy groups, which gives a clue to understanding of Frobenius manifold
structure on the orbit space.
It was shown by B.Dubrovin (\cite{Du1996,Du1998}) that any finite Coxeter group can serve as
a monodromy group of a polynomial Frobenius manifold, that is to say, the potential is
a polynomial with respect to the flat coordinates $t^1,\cdots,t^l$. Furthermore, he put forward the
following conjecture,
{\it ``Any massive polynomial Frobenius manifold with positive invariant
 degrees is isomorphic to the orbit space of a finite Coxeter group",}
 which was proved by C.Hertling (\cite{Her2002}). Besides these, in \cite{Ber3,Z2007} it has been
shown that there are $l$ different Frobenius manifold structures on the orbit spaces of the
Coxeter groups $B_l$ and $D_l$. Especially, we also proved in \cite{Z2007} that the corresponding potentials are
meromorphic along the divisors $t^k=0$ for a given $1\leq k\leq l-2$, where $t^1,\dots,
t^{l}$ are the flat coordinates of the Frobenius manifold.

Let $R$ be an irreducible reduced root system defined in an
$l$-dimensional Euclidean space $V$ with Euclidean inner product
$(~,~)$, we fix a basis of simple roots $\al_1,\dots,\al_l$ and
denote by  $\alpha_j^{\vee},\ j=1,2,\cdots,l$ the corresponding
coroots. The Weyl group $W(R)$ is generated by the reflections
$
{\bx}\mapsto \bx-(\al_j^\vee,\bx)\al_j,\quad \forall\, {\bx}\in V,\ j=1,\dots,l.\nn
$
 The semi-direct product of $W(R)$ by the lattice of coroots yields the
affine Weyl group $W_a(R)$ that acts on $V$ by the affine
transformations \beq \bx\mapsto w(\bx)+\sum_{j=1}^l m_j\av_j,\quad
w\in W,\ m_j\in \mathbb{Z}. \nn \eeq
We denote by $\omega_1,\dots,\omega_l$ the fundamental weights  defined by the relations
\beq
(\omega_i,\al_j^\vee)=\delta_{ij},\quad i, j=1,\dots,l.\nn
\eeq

B.Dubrovin and Y.Zhang in \cite{DZ1998} (also \cite{DSZZ2019}) defined  {\it an extended affine Weyl group $\tw^{(k)}(R)$}, which acts on the extended space
$V\oplus \mathbb{R}$
and is generated by the transformations
\begin{equation}
x=(\bx,x_{l+1})\mapsto (w(\bx)+\sum_{j=1}^l m_j\av_j, \
x_{l+1}),\quad w\in W,\ m_j\in \mathbb Z,\nn
\end{equation}
and
\begin{equation}
x=(\bx,x_{l+1})\mapsto (\bx+\gamma\, \w_k,\ x_{l+1}-\gamma).\nn
\end{equation}
Here $1\le k\le l$, $\gamma=1$ except for the cases when $R=B_l, k=l$ and $R=F_4, k=3$ or $k=4$, in these three
cases $\gamma=2$. For a particular choice of a simple root $\al_k$, they proved an analogue of Chevalley theorem for their invariants.
On the orbit space of $\widetilde{W}^{(k)}(R)$, they constructed a Frobenius mainfold structure
whose potential is a weighted homogeneous polynomial of $t^1,\cdots, t^{l+1}, e^{t^{l+1}}$, where $t^1,\dots,
t^{l+1}$ are the flat coordinates of the Frobenius manifold.

 Observe that for the root system of type $A_l$, there is
in fact no restrictions on
the choice of $\al_k$. However, for the root systems of
type $B_l, C_l$, $D_l, E_6$, $E_7$, $E_8$, $F_4$, $G_2$ there is
only one choice for each. In \cite{SL1997} Slodowy pointed out that
the Chevalley type theorem is a consequence of the results of
Looijenga and Wirthm\"uller \cite{LO1,LO2,W1986}, and in fact it holds true for
any choice of the base element $\al_k$. A natural question is that
{\it `` Whether the geometric structures that were revealed by Dubrovin-Zhang's construction
 also exist on the orbit spaces of the extended affine
Weyl groups for an arbitrary choice of $\al_k$?"}
 Our recent work in \cite{DSZZ2019} is to give an affirmative
answer to this question for the root systems of type $B_l, C_l$ and also for $D_l$. We show, by fixing
another integer $0\le m\le l-k$, that on the corresponding orbit spaces there also exist Frobenius manifold
structures with potentials $F(t)$ that are weighted homogeneous polynomials  w.r.t $t^1,\cdots, t^{l+1},
\frac{1}{t^{l-m}}, \frac1{t^l}$, $e^{t^{l+1}}$. We also construct Landau--Ginzburg (briefly LG) superpotentials for these
Frobenius manifold structures.

In this paper we will present a new extension of affine Weyl groups denoted by $\tw^{(k,k+1)}(R)$, which is different from 
those in \cite{DZ1998,DSZZ2019}, and study Frobenius manifold structures on the corresponding orbit spaces $\mathcal{M}^{(k,k+1)}(R)$, where
the new extended affine Weyl groups $\widetilde{W}^{(k,k+1)}(R)$ act on the extended space
$V\oplus\mathbb{R}^2$ generated by the transformations
\beq\begin{array}{c}
 x=({\bf x},~x_{l+1},x_{l+2})\mapsto
 (w({\bf x})+\dsum_{j=1}^lm_j\alpha_j^{\vee},~x_{l+1},x_{l+2}),\quad
 w\in W,\quad m_j\in \mathbb{Z},\end{array}\nn \eeq
and
\begin{equation}
x=({\bf x},~x_{l+1},x_{l+2})\mapsto (\bx+\gamma\, \w_k,\ x_{l+1}-\gamma,\ x_{l+2})\nn
\end{equation}
and
\begin{equation}
x=({\bf x},~x_{l+1},x_{l+2})\mapsto (\bx+\gamma\, \w_{k+1},\ x_{l+1},\ x_{l+2}-\gamma)\nn
\end{equation}
Here $1\le k\le l-1$, $\gamma=1$ except for the cases when $R=B_l, k=l$ and $R=F_4, k=3$ or $k=4$, in these three
cases $\gamma=2$.

By a direct verification, we could not obtain any flat pencil of metrics and  Frobenius manifold structures
on the orbit spaces $\mathcal{M}^{(1,2)}(B_2)$,\, $\mathcal{M}^{(1,2)}(C_2)$ and $\mathcal{M}^{(1,2)}(G_2)$ etc.
We thus have to restrict our study to the type $A_l$ case, i.e. $\mathcal{M}^{(k,k+1)}(A_l)$ and will show that (see Theorem \ref{thm4.7})

\begin{Mthm}\label{mt1}
 For any fixed integer $1\le k<l$, there exists a unique Frobenius manifold structure
 of charge $d=1$ on the orbit
space $\mathcal{M}^{(k,k+1)}(A_l)\setminus\{\tilde{y}_{l+1}=0\}\cup\{\tilde{y}_{l+2}=0\}$ of $\widetilde{W}^{(k,k+1)}(A_l)$ such that

\begin{enumerate}
\item the invariant flat metric and the intersection form of the Frobenius manifold structure coincide
with the metrics $(\eta^{ij}(y))$ in \eqref{dfeta} and $(g^{ij}(y))$ in \eqref{df3.3} respectively;
\item the unity and the Euler vector fields have the form
\beq e=\dfrac{\p}{\p y^k}+\dfrac{\p}{\p y^{k+1}}\label{df1.1} \eeq
and
\begin{equation}E=\dsum_{\alpha=1}^{l}d_\alpha y^\alpha \dfrac{\p}{\p y^\alpha}
+\dfrac{1}{k}\dfrac{\p}{\p y^{l+1}}+\dfrac{1}{l-k}\dfrac{\p}{\p y^{l+2}}~,\label{df1.2}
\end{equation}
where $d_1,\dots,d_{l}$ are defined in \eqref{df2.11};

\item in the flat coordinates $t^1$, \dots, $t^{l+2}$ of the metric \eqref{dfeta} defined on certain covering of $\mathcal{M}^{(k,k+1)}(A_l)$
the potential of the Frobenius manifold structure is of the form  $F(t)=\widehat{F}(t)+\frac{1}{2}(t^{k+1})^2\log(t^{k+1})$, where $\widehat{F}(t)$ is
a weighted homogeneous polynomial in $t^1,t^2$, $\cdots, t^{l+2}$, $e^{t^{l+1}}$, $e^{t^{l+2}-t^{l+1}}$.\end{enumerate}
\end{Mthm}

On the orbit space of the extended affine Weyl group $\widetilde{W}^{(k)}(A_l)$, an alternative construction of the Frobenius manifold
structure was given in \cite{DZ1998}. This structure was given in terms of a LG superpotential construction. In particular, it was shown that
$\widetilde{W}^{(k)}(A_l)$ describes the monodromy of roots of trigonometric polynomials - the superpotential - with a given
bidegree being of the form
\beq \lambda(\varphi)=e^{{\bf i} k\varphi}+a_1e^{{\bf i} (k-1)\varphi}+\cdots+a_{l+1}e^{{\bf i} (k-l-1)\varphi},\quad a_{l+1}\ne 0.\nn\eeq
A natural question is that

``{\it Whether a similar construction about the Frobenius manifold structure exists  on the orbit space $\mathcal{M}^{(k,k+1)}(A_l)\setminus\{\tilde{y}_{l+1}=0\}\cup\{\tilde{y}_{l+2}=0\}$ of
 the extended affine Weyl group $\widetilde{W}^{(k,k+1)}(A_l)$?}\,"

 Let $\mathbb{M}_{k,l-k+1,1}$ be the space of a particular class of LG superpotentials consisting of trigonometric-Laurent series of one variable with tri-degree $(k+1,l-k,1)$,
these being functions of the form
$$\lambda(\varphi)=(e^{{\bf i}\varphi}-a_{l+2})^{-1}(e^{{\bf i}(k+1)\varphi}+a_1e^{{\bf i}k\varphi}+\cdots+a_{l+1} e^{{\bf i}(k-l)\varphi}), \quad
 a_{l+1}a_{l+2}\ne 0,$$
where $a_j \in \mathbb{C}$ for $j=1,\cdots, l+2$.
The space $\mathbb{M}_{k,l-k+1,1}$ carries a natural structure of Frobenius manifold. Its invariant inner product
$\eta$ and the intersection form $g$ of two vectors $\p'$, $\p''$
tangent to $\mathbb{M}_{k,l-k+1,1}$ at a point $\lambda(\varphi)$ can be defined by
the formulae \eqref{AM2} and \eqref{AM3}.
We will show that (see Theorem\,\ref{Main2})

\begin{Mthm}\label{mt2}
 The Frobenius manifolds $\mathcal{M}^{(k,k+1)}(A_l)\setminus\{\tilde{y}_{l+1}=0\}\cup\{\tilde{y}_{l+2}=0\}$ and  $\mathbb{M}_{k,l-k+1,1}$ are locally isomorphic.
\end{Mthm}

\section{A new class of extended affine Weyl groups $\widetilde{W}^{(k,k+1)}(A_l)$ }
\label{sec-2}

To keep self-contained, we recall some known facts about Weyl groups of type $A_l$,
see \cite{Bour} for details. Let $\mathbb{R}^{l+1}$ be a $(l+1)$-dimensional Euclidean space with Euclidean inner product
$(~,~)$ and an orthonormal basis
$\ep_1,\cdots,\ep_{l+1}$. Let $A_l$ be an irreducible
 reduced root system in the hyperplane $V=\left\{ \left. \dsum_{s=1}^{l+1}v_s\ep_s\in \mathbb{R}^{l+1}\right |\dsum_{s=1}^{l+1} v_{s}=0\right\}.$
 We fix a basis
 $$\al_1=\ep_1-\ep_2\, ,\cdots,\, \al_l=\ep_l-\ep_{l+1} $$
of simple roots. The corresponding coroots are
$\al_j^{\vee}=\al_j$ for $j=1,\cdots,l$.
The Weyl group  $W=W(A_l)$ is generated by the reflections
 \beq {\bx}\mapsto \bx-(\al_j^\vee,\bx)\al_j,\quad \forall\, {\bx}\in V,\
j=1,\dots,l. \eeq $W$ acts on $V$ by permutations of the
coordinates $v_1,\cdots,v_{l+1}$. The basic $W$-invariant Fourier
polynomials coincide with the elementary symmetric functions
 \beq
y_j(\bx)=\sigma_j(e^{2\pi i v_1},\cdots, e^{2\pi i v_{l+1}}),
\quad j=1,\cdots,l.\label{df1.2}\eeq

\begin{de}For any fixed integer $1\leq k<l$, we call $\widetilde{W}=\widetilde{W}^{(k,k+1)}(A_l)$ to be
an extended affine Weyl group of type $A$ if it acts on
\begin{equation*}\widetilde{V}= V\oplus\mathbb{R}^2\end{equation*} generated by the transformations
\beq\begin{array}{c}
 x=({\bf x},~x_{l+1},x_{l+2})\mapsto
 (w({\bf x})+\dsum_{j=1}^lm_j\alpha_j^{\vee},~x_{l+1},x_{l+2}),\quad
 w\in W,\quad m_j\in \mathbb{Z},\end{array}\eeq
 and
 \beq x=({\bf x},~x_{l+1},x_{l+2})\mapsto
({\bf x}+\omega_{k},~x_{l+1}-1, x_{l+2}),\label{df2.2} \eeq and
\beq x=({\bf x},~x_{l+1},x_{l+2})\mapsto ({\bf x}
+\omega_{k+1},~x_{l+1},x_{l+2}-1). \label{df2.3}\eeq
\end{de}

Coordinates $x_1,\cdots, x_l$ may be introduced on the space $V$ via the expression
 \beq {\bf x}=x_1\al_1^{\vee}+\cdots+ x_l\al_l^{\vee}.\eeq
 That is to say,
\beq v_1=x_1, \quad v_i=x_j-x_{j-1},\quad v_{l+1}=-x_{l},\quad ~j=2,\cdots, l.\eeq

\begin{de}\label{def2.1}$\mathcal{A}=\mathcal{A}^{(k,k+1)}(A_l)$ is the ring of
all $\widetilde{W}$-invariant Fourier polynomials of
$x_1,\cdots,x_l$, $\frac{1}{l+1}x_{l+1}$, $\frac{1}{l+1}x_{l+2}$ that
are bounded in the following limit conditions
\begin{equation} {\bf x}={\bf x}^0-i\omega_{k}\tau,~
x_{l+1}=x_{l+1}^0+i\tau, x_{l+2}=x_{l+2}^0,~\tau\to+\infty
\label{df2.6}\end{equation} and
 \begin{equation} {\bf x}={\bf
x}^0-i\omega_{k+1}\tau,~x_{l+1}=x_{l+1}^0, ~x_{l+2}=x_{l+2}^0+i\tau,~
\tau\to+\infty \label{df2.7}\end{equation}
 for any $x^0=({\bf x}^0,x_{l+1}^0,x_{l+2}^0)$.
\end{de}

Conditions \eqref{df2.6} and \eqref{df2.7} are essential for this construction as did in \cite{DZ1998,DSZZ2019}.
For simplicity, we introduce a set of numbers
\begin{equation}d_{j,k}:=(\omega_j,\omega_k)=\left\{
\begin{array}{ll}
\frac{j(l-k+1)}{l+1},\quad j=1,\cdots, k,\\
\frac{k(l-j+1)}{l+1},\quad j=k+1,\cdots, l
\end{array}\right.\label{df2.8}
\end{equation}
and define the following Fourier polynomials
\begin{equation}\begin{array}{l}\widetilde{y}_j(x)=e^{2\pi i(d_{j,k}\,x_{l+1}+d_{j,k+1}\,x_{l+2})}y_j({\bf
x}),\quad j=1,\cdots, l,\\
\widetilde{y}_{l+1}(x)=e^{2\pi ix_{l+1}},\quad
\widetilde{y}_{l+2}(x)=e^{2\pi ix_{l+2}}. \end{array}\label{df2.9}\end{equation}

\begin{lem}\label{lem2.3}(\cite{DZ1998}) For any fixed integer $1\leq k_r\leq l$, we have
\begin{equation}y_j({\bf x})=e^{2\pi d_{j,k_r}\tau}[y_{j}^{0,r}({\bf x^0})+\mathcal{O}
(e^{-2\pi\tau})], \quad \tau\to+\infty, \quad \mbox{for $j=1,\cdots,l$},\nn \end{equation} where ${\bf x}={\bf x}^0-i\omega_{k_r}\tau$  and
\begin{equation*}
y_j^{0,r}({\bf x^0})=\dfrac{1}{n_j}\dsum_{\tiny{w\in W,(w(\omega_j)-\omega_j,\omega_{k_r})=0}} e^{2\pi
i(w(\omega_j),~{\bf x^0})},\end{equation*}
where $n_j={\#\{w\in W|e^{2\pi i(\omega_j,w({\bf x}))}=e^{2\pi i(\omega_j,{\bf x})}\}}.$
Moreover, the Fourier polynomials
 $y_1^{0,r}({\bf x^0}),\cdots,y_l^{0,r}({\bf x^0})$ are algebraically independent.\end{lem}

From these explicit expressions in \eqref{df2.9} and Lemma \ref{lem2.3}, it is not difficult to see that
$\widetilde{y}_j(x)\in\mathcal{A}$ for $j=1,\cdots, l+2$. Furthermore, we have

 \begin{thm}(Chevalley-type theorem)
\label{thm2.4}The ring $\mathcal{A}$ is isomorphic to the ring of
polynomials of $\widetilde{y}_1(x),\cdots,\widetilde{y}_{l+2}(x)$.\end{thm}

\begin{proof}  Observe that $\widetilde{y}_1(x),\cdots,\widetilde{y}_{l+2}(x)$
 are algebraically independent. So in order to prove the theorem, we only need to
show that any element $f(x)$ of the ring $\mathcal{A}$ can be represented as a polynomial of
$\widetilde{y}_1(x),\cdots, \widetilde{y}_{l+2}(x)$. From the invariance with respect to $\widetilde{W}$, it follows
that  $f(x)$ can be represented as a polynomial of
$\widetilde{y}_1(x),\cdots, \widetilde{y}_{l+1}(x)$, $\widetilde{y}_{l+2}(x)$,
$\widetilde{y}^{-1}_{l+1}(x)$, $\widetilde{y}^{-1}_{l+2}(x)$. It suffices to
show that in $f(x)$ there are no negative powers of
$\widetilde{y}_{l+1}(x)$ and $\widetilde{y}_{l+2}(x)$.

Assume that
\begin{equation}
f(x)=\dsum_{s\geq-S}\widetilde{y}_{l+1}^s\dsum_{t\in \Lambda}\widetilde{y}_{l+2}^t\,Q_{s,t}(\widetilde{y}_1(x),\cdots,
\widetilde{y}_l(x))
\end{equation}
for a positive integer $S$ and the polynomial $Q_{-S, t_0}(\widetilde{y}_1(x),\cdots,\widetilde{y}_l(x))$ does not vanish identically,
where $t_0=\mbox{min} \{t\in \Lambda| \mbox{ $Q_{-S,t}$ does not vanish identically}\}$ and $\Lambda$ is a finite subset of $\mathbb{Z}$.
With the use of Lemma \ref{lem2.3}, in the limit \eqref{df2.6} the function $f(x)$ behaves as
\eqa
f(x)=e^{2\pi S\tau-2\pi iSx_{l+1}^0}\,
\dsum_{t\in \Lambda}e^{2\pi itx_{l+2}^0}\,[Q_{s,t}(\widetilde{y}_1^{0,1}(x^0),\cdots,\widetilde{y}_l^{0,1}(x^0))+\mathcal{O}(e^{-2\pi\tau})],\nn
\eeqa where $\widetilde{y}_j^{0,1}(x^0)=e^{2\pi i
(d_{j,k}x_{l+1}^0+d_{j,k+1}x_{l+2}^0)}y_j^{0,1}({\bf x^0})$ for $j=1,\cdots,l$.
In order to assure the function $f(x)$ bounded for $\tau\mapsto +\infty$,
it is necessary to have
\begin{equation}Q_{-S,t_0}(\widetilde{y}_1^{0,1}(x^0),\cdots,\widetilde{y}_l^{0,1}(x^0))
\equiv 0,\nn \end{equation} which is a contradiction with the algebraic independence of $\widetilde{y}_1^{0,1}(x^0), \cdots,
\widetilde{y}_l^{0,1}(x^0)$. This means that
there are no negative powers of $\widetilde{y}_{l+1}(x)$. Similarly one can show that
there are no negative powers of $\widetilde{y}_{l+2}(x)$. This completes the proof of the
theorem.\end{proof}

\begin{cor}\label{cor2.6} The function $\deg$  defined as
\beq \begin{array}{l}
\deg \widetilde{y}_{l+1}=\dfrac{1}{k},\quad \deg \widetilde{y}_{l+2}=\dfrac{1}{l-k},\\
d_j:=\deg \widetilde{y}_j=\dfrac{d_{j,k}}{k} +\dfrac{d_{j,k+1}}{l-k}=\left\{\begin{array}{cl}
\frac{j}{k},& j=1,\cdots, k,\\
 \frac{l-j+1}{l-k}, & j=k+1,\cdots, l
\end{array}\right.\end{array} \label{df2.11}\eeq
determines on $\mathcal{A}$ a structure of graded polynomial ring. Especially,
\beq d_k=d_{k+1}=1> d_s, \quad s\ne k, k+1.\label{df2.12}\eeq
\end{cor}

The numbers $d_1,\dots,d_{l+2}$ with $d_{l+1}=d_{l+2}=0$ satisfy a duality relation. For any given integer $k$, we
denote $A_l\setminus \{\alpha_{k},\alpha_{k+1}\}={\mathcal R}_{1}\cup {\mathcal R}_{2}$, where
 ${\mathcal R}_{1}=\{\alpha_{1},\cdots,\alpha_{k-1}\}$ and
 ${\mathcal R}_{2}=\{\alpha_{k+2},\cdots,\alpha_{l}\}$.
On each component we have an involution $j\mapsto j^*$ given by the
reflection with respect to the center of the component. Let us define
\beq
k^*=l+1,~~{(k+1)}^*=l+2,~~(l+2)^*=k+1,~~(l+1)^*=k,\label{df2.13}\eeq
then
\beq d_j+d_{j^*}=1,\quad j=1,\dots, l+2. \label{df2.14}\eeq

\section{A flat pencil of metrics on the orbit space $\mathcal{M}$}\label{sec-3}

Let us denote
$\mathcal{M}^{(k,k+1)}:=\widetilde{V}\otimes\mathbb{C}/\widetilde{W}$, called
the $orbit$ $space$ of the extended Weyl group
$\widetilde{W}$. We define an indefinite flat metric $(dx_i,dx_j)^{\thicksim}$ on
$\widetilde{V}_\mathbb{C} = \widetilde{V} \otimes_\mathbb{R} \mathbb{C}$
where
$\widetilde V$ is the orthogonal direct sum of $V$ and $\mathbb
R^2$.  Here ${V}$ is endowed with the $W$-invariant Euclidean metric \footnote{{As is common in the
Frobenius manifold literature, we use the word metric to denote
a complex-valued, symmetric, non-degenerate, bilinear form.}}
\beq (dx_a,dx_b)^\thicksim=\frac{1}{4\pi^2}(\omega_a,\omega_b),\quad 1\leq a,b\leq l \label{df3.01} \eeq
and $\mathbb R^2$ is endowed with the metric
\beq
(dx_{l+1},dx_{l+1})^\thicksim=-\frac{\tau_{11}}{4\pi^2},\quad
  (dx_{l+1},dx_{l+2})^\thicksim=-\frac{\tau_{12}}{4\pi^2},\quad (dx_{l+2},dx_{l+2})^\thicksim= -\frac{\tau_{22}}{4\pi^2},\label{df3.02}
\eeq
where
\beq\label{df3.3}
\left(\begin{array}{ll}
\tau_{11}& \tau_{12}\\
\tau_{12} & \tau_{22}\end{array}\right)=\left(\begin{array}{ll}
d_{k,k} & d_{k,k+1}\\
d_{k+1,k} & d_{k+1,k+1}\end{array}\right)^{-1}=\left(\begin{array}{cc}
\frac{k+1}{k} & -1\\
-1 & \frac{l-k+1}{l-k}\end{array}\right).\nn
\eeq

The set of generators for the ring ${\mathcal A}$ are defined by \eqref{df2.9}. They form a system of global coordinates on
${\mathcal M}^{(k,k+1)}$. We now introduce a system of
local coordinates on ${\mathcal M}^{(k,k+1)}$ as follows
\beq\label{zh9}
y^1=\ty_1,\dots, y^l=\ty_l,\ y^{l+1}=\log \ty_{l+1}=2 \pi i
x_{l+1},\ y^{l+2}=\log \ty_{l+2}=2 \pi i
x_{l+2}.
\eeq
They live on the universal covering $\widetilde{\mathcal M}$ of $\mathcal{M}$, where
$\mathcal{M}:={\mathcal M}^{(k,k+1)}\setminus\{\ty_{l+1}=0\}\cup\{\ty_{l+2}=0\}$.
 The projection
\beq
\Pr: \widetilde V\to \widetilde{\mathcal M}, \qquad (x_1,\cdots,x_{l+2})\mapsto (y^1,\cdots,y^{l+2})
\label{df3.2}\eeq
induces a symmetric
bilinear form on $T^{*}{\widetilde{\mathcal{M}}}$
\begin{equation}
(d y^i,d y^j)^{\sptilde}\equiv
g^{ij}(y):=\dsum_{a,b=1}^{l+2}\dfrac{\p y^i}{\p x_a}\dfrac{\p
y^j}{\p x_b}(dx_a,dx_b)^{\sptilde}.\label{df3.3}
\end{equation}

 \begin{lem} \label{lem3.1} The matrix entries $g^{ij}(y)$ of \eqref{df3.3}
are weighted homogeneous polynomials in $y^1,\cdots, y^{l}$, $e^{y^{l+1}}$, $e^{y^{l+2}}$
of the degree
$\deg g^{ij}(y)=\deg y^i+\deg y^j,$
here $\deg y^{l+1+\nu}=d_{l+1+\nu}=0$ for $\nu=0,1$.
The matrix $\left(g^{ij}(y)\right)$ does not degenerate outside the \mbox{Pr}-images of the hyperplanes
\beq\left\{({\bf x},x_{l+1},x_{l+2})|
(\beta ,{\bf x})=m\in\mathbb{Z}, \forall{x_{l+1}}, \forall{x_{l+2}}
\right\}, \quad \beta\in {\Phi}^+, \nn\eeq
where ${\Phi}^+$ is the set of the all positive roots.
\end{lem}
\begin{proof}With the use of \eqref{df3.3},\eqref{df2.8} and \eqref{df2.11}, we obtain
\beq\begin{array}{l}
g^{j,l+1}(y)=\zeta_j d_jy^j, \quad g^{j,l+2}(y)=(1-\zeta_j) d_jy^j, \quad j=1,\cdots, l,\\
g^{l+1,l+1}(y)=\frac{k+1}{k},\quad  g^{l+1,l+2}(y)=-1,\quad  g^{l+2,l+2}(y)=\frac{l-k+1}{l-k},\label{df3.4}
\end{array}\eeq
where $\zeta_j=\left\{\begin{array}{ll} 1, & 1\leq j\leq k,\\
0, & k+1\leq j\leq l. \end{array}\right.$
Also, for $1\leq i,j\leq l$ we have
\beq g^{ij}(y)=c_{ij}y^iy^j+\frac{1}{4\pi^2}\dsum_{p,q=1}^{l}\dfrac{\p y^i}{\p x_p}\dfrac{\p
y^j}{\p x_q}(\omega_p,\omega_q),\quad c_{ij}=(\zeta_j d_{i,k}+(1-\zeta_j)d_{i,k+1})d_j,\label{dff3.5} \eeq
which are Fourier  polynomials invariant with respect to $\widetilde{W}$ and bounded in the limits \eqref{df2.6} and \eqref{df2.7}. It follows from
Theorem \ref{thm2.4} and \eqref{df3.3} that $g^{ij}(y)$ are weighted homogeneous polynomials
in $y^1,\cdots, y^{l}$, $e^{y^{l+1}}$, $e^{y^{l+2}}$ of the degree $\deg g^{ij}(y)=\deg y^i+\deg y^j$.

Observe that the Jacobian of the projection map \eqref{df3.2} is given by
\eqa
\det\left(\dfrac{\p y^j}{\p x_a}\right)&=&-4\pi^2e^{2\pi i \sum_{j=1}^l(d_{j,k}x_{l+1}+d_{j,k+1}x_{l+2})}\det\left(\dfrac{\p y_j({\bf x})}{\p x_p}\right)\nn\\
&=&c\, e^{2\pi i \sum_{j=1}^l(d_{j,k}x_{l+1}+d_{j,k+1}x_{l+2})} J({\bf x}), \label{df3.6}
\eeqa
where $J({\bf x})=e^{-\sum_{\alpha\in \Phi^+} {\pi i (\alpha,{\bf x})}}\displaystyle{\prod_{\beta\in\Phi^+}}(e^{2\pi i (\beta,{\bf x})}-1)$ and
$c$ is a nonzero constant (\cite{Bour}). So the projection map \eqref{df3.2} is  a local diffeomorphsim outside the
above hyperplanes, which assures the nondegeneracy of $(g^{ij}(y))$.
\end{proof}

\begin{lem}\label{lem3.2}For $k\leq i, j\leq k+1$, the term $y^iy^j$ only possibly appears in $g^{ij}(y)$ and $g^{ji}(y)$ with the coefficient $c_{ij}-d_{i,j}$, where
$c_{ij}=(\zeta_jd_{i,k}+(1-\zeta_j)d_{i,k+1})d_j$ and $d_{i,j}=(\omega_i,\omega_j)$.\end{lem}

\begin{proof}
From \eqref{dff3.5}, we have for $k\leq i,j\leq k+1$
\beq g^{ij}(y)=e^{2\pi i [(d_{i,k}+d_{j,k})x_{l+1}+(d_{i,k+1}+d_{j,k+1})x_{l+2}]} (c_{ij}y_i({\bf x})y_j({\bf x})+\beta_{ij}({\bf x}))
,\label{df3.5} \eeq
where
\eqa\beta_{ij}({\bf x})&=&\frac{1}{4\pi^2}
\dsum_{p,q=1}^{l}\dfrac{\p y_i({\bf x})}{\p x_p}\dfrac{\p
y_j({\bf x})}{\p x_q}(\omega_p,\omega_q)\nn\\
&=&-\frac{1}{n_in_j}\dsum_{w,w'\in W}
e^{2\pi i (w(\omega_i)+w'(\omega_j),{\bf x})}
(w(\omega_i),w'(\omega_j).\nn \eeqa
Now we use the standard partial ordering of the weights (see the page 69 in \cite{Hum})
$$\omega \succ \omega' \quad \mbox{iff}\quad \omega-\omega'=\dsum_{m=1}^l c_m \alpha_m$$
for some nonnegative  integers $c_1,\cdots, c_l$. In this case, we will write
$e^{2\pi i (\omega, {\bf x})}\succ e^{2\pi i (\omega', {\bf x})}.$
All the terms in the $W$-invariant Fourier polynomials $\beta_{ij}({\bf x})$ are strictly less than
$e^{2\pi i (\omega_i+\omega_j, {\bf x})}$ except the terms $-d_{i,j} e^{2\pi i (\omega_i+\omega_j, {\bf x})}$.
So the term $y^iy^j$ possibly appears in $g^{ij}(y)$ and $g^{ji}(y)$  with the coefficient $c_{ij}-d_{i,j}$.

Observe that
$$ \frac{g^{kk}(y)}{y^ky^{k+1}}=e^{2\pi i \left(\frac{k}{l+1}x_{l+1}-\frac{l-k}{l+1}x_{l+2}\right)}
\left(\dfrac{c_{kk}y_k({\bf x})}{y_{k+1}({\bf x})}+\dfrac{\beta_{kk}({\bf x})}{y_k{(\bf x)}y_{k+1}({\bf x})}\right)
$$
and
$$\frac{g^{kk}(y)}{(y^{k+1})^2}=e^{2\pi i \left(\frac{k}{l+1}x_{l+1}-\frac{l-k}{l+1}x_{l+2}\right)}
\dfrac{c_{kk}(y_k({\bf x}))^2+\beta_{kk}({\bf x})}{(y_{k+1}({\bf x}))^2}.$$
So $y^{k}y^{k+1}$ and $(y^{k+1})^2$ do not appear in $g^{kk}(y)$. Similarly, we could prove the other cases.

\end{proof}

\begin{lem}\label{alem3.3} Denote
\beq e=\varsigma_1 \frac{\p}{\p y^{k}}+\varsigma_2 \frac{\p}{\p y^{k+1}}, \quad \varsigma_1,\varsigma_2\in\mathbb{R},\eeq
then for $1\le i,j\le l+2$,
\beq \mathcal{L}_e(\mathcal{L}_e g^{ij}(y))=0,\label{ddf3.11}\eeq
where $\mathcal{L}_e$ is the Lie derivative along
the vector field $e$.
\end{lem}

\begin{proof}According to the weighted homogeneity and \eqref{df2.12} and \eqref{df3.4},
it suffices to show that $k\leq i,j\leq k+1$,
\beq \dfrac{\p^2}{\p y^k \p y^{k}}{g^{ij}(y)}=\dfrac{\p^2}{\p y^{k+1} \p y^{k+1}}{g^{ij}(y)}=\dfrac{\p^2}{\p y^k \p y^{k+1}}{g^{ij}(y)}=0.\label{ddf3.12}\eeq
It follows from
$$ c_{k(k+1)}=d_{k,k+1},\quad c_{(k+\nu)(k+\nu)}=d_{k+\nu,k+\nu}$$
that $g^{k+\nu,k+\nu}(y)$ does not contain $(y^{k+\nu})^2$ for $\nu=0,1$, and $g^{k,k+1}(y)$ ($=g^{k+1,k}(y)$) does not contain $y^ky^{k+1}$.
Combining with Lemma \ref{lem3.2}, we obtain the desired \eqref{ddf3.12} and complete the proof of this lemma.
\end{proof}

\begin{cor}\label{cor3.4}
For $1\leq i,j\leq l+2$, $$g^{ij}(\cdots, y^{k}+\varsigma_1\lambda,y^{k+1}+\varsigma_2\lambda,\cdots)$$
are linear in the parameter $\lm$.
\end{cor}

Suppose $\Sigma$ is the discriminant of $\mathcal{M}^{(k,k+1)}$, i.e.,
$\Sigma=\{y|\det(g^{ij}(y))=0\}$, then on $\mathcal{M}^{(k,k+1)}\setminus \Sigma$ the matrix $(g^{ij}(y))$ is
invertible. The inverse matrix
$(g_{ij}(y))$ determines a flat metric on $\mathcal{M}^{(k,k+1)}\setminus \Sigma$.
Let us now compute the coefficients of the correspondent Levi-Civita connection $\nabla$
for the metric $g_{ij}(y)$. It is convenient to consider the
contravariant components of the connection
\beq
\Gamma_{m}^{ij}(y)=(dy^i, \nabla_m dy^j),\nn \eeq
which are related to
the standard Christoffel coefficients by the formula
\beq
\Gamma_{m}^{ij}(y)=-g^{is}(y)\Gamma_{sm}^{j}(y).\nn \eeq
For the contravariant components, we have
the following formulae
\beq
\Gamma_m^{ij}(y)dy^m=\dfrac{\p y^i}{\p x_a}\dfrac{\p^2 y^j}{\p x_b \p x_r}(dx_a, dx_b)^\thicksim dx_r
\label{df3.13}\eeq
and
 \beq
2g^{sm}(y)\Gamma_m^{ij}(y)=g^{im}(y)\dfrac{\p g^{js}(y)}{\p
y^m}+g^{sm}(y)\dfrac{\p g^{ji}(y)}{\p y^m} -g^{jm}(y)\dfrac{\p
g^{is}(y)}{\p y^m}\,\label{df3.14}\eeq
and
\beq \dfrac{\p g^{ij}(y)}{\p y^m}=\Gamma_m^{ij}(y)+\Gamma_m^{ji}(y). \label{df3.15}\eeq

\begin{lem}\label{lem3.5} $\Gamma_{m}^{ij}(y)$ are weighted homogeneous polynomials in $y^1,\cdots, y^{l}$, $e^{y^{l+1}}$, $e^{y^{l+2}}$
of the degree $\deg \Gamma_{m}^{ij}(y)=\deg y^i+\deg y^j-\deg y^m.$
\end{lem}
\begin{proof} By using \eqref{df3.13} and \eqref{df3.6}, we can represent
$$(\Gamma_1^{ij}(y),\cdots, \Gamma_{l+2}^{ij}(y))\left(\dfrac{\p y^r}{\p x_a}\right)
=\left(\dfrac{\p y^i}{\p x_a}\dfrac{\p^2 y^j}{\p x_b \p x_1}(dx_a, dx_b)^\thicksim,
\cdots,\dfrac{\p y^i}{\p x_a}\dfrac{\p^2 y^j}{\p x_b \p x_{l+2}}(dx_a, dx_b)^\thicksim\right),$$
and
\beq \Gamma_m^{ij}(y)=e^{2\pi i \sum_{\nu=0}^1(d_{i,k+\nu}+d_{j,k+\nu}-d_{m,k+\nu})x_{l+1+\nu}} \dfrac{P_m^{ij}({\bf x})}{J({\bf x})},\label{df3.16}\eeq
where $d_{l+1+\nu,k+\nu}=0$ for $\nu=0,1$ and $P_m^{ij}({\bf x})$ is certain Fourier polynomial in $x_1,\cdots,x_l$.
 As discussed the Lemma 2.2 in \cite{DZ1998}, $P_m^{ij}({\bf x})$ is
anti-invariant with respect to Weyl group $W$ and divisible by $J({\bf x})$. We thus knows $\Gamma_{m}^{ij}(y)\in \mathcal{A}$,
whose homogeneity property is obvious.
\end{proof}

\begin{lem}\label{lem3.6}
For $1\leq i,j\leq l+2$, we have
\beq \mathcal{L}_e(\mathcal{L}_e \Gamma_m^{ij}(y))=0.\eeq
 Equivalently, $\Gamma_m^{ij}(\cdots, y^{k}+\varsigma_1\lambda,y^{k+1}+\varsigma_2\lambda,\cdots)$
are linear in the parameter $\lm$.
\end{lem}

\begin{proof}By the degrees, it suffices to show that
$$\mathcal{L}_e(\mathcal{L}_e \Gamma_{l+i}^{k+\nu,\ k+t}(y))=0,\quad i=1,2,\quad \nu,t=0,1.$$

Observe that
 \eqa \Gamma_{l+i}^{k+\nu,k+\nu}(y)&=&\dfrac{\p x_r}{\p y^{l+i}}\dfrac{\p y^{k+\nu}}{\p x_p}
 \dfrac{\p^2 y^{k+\nu}}{\p x_q \p x_r}(dx_p, dx_q)^\thicksim \nn \\
 &=&\dfrac{\p y^{k+\nu}}{\p x_p}\dfrac{\p}{\p y^{l+i}}(\dfrac{\p y^{k+\nu}}{\p x_q})(dx_p, dx_q)^\thicksim \nn\\
&=&\frac{1}{2} \frac{\p }{\p y^{l+i}}g^{k+\nu,k+\nu}(y),\nn\eeqa
 and using \eqref{ddf3.11}, then
 $$\mathcal{L}_e(\mathcal{L}_e \Gamma_{l+i}^{k+\nu,\ k+\nu}(y))=0$$
  for $\nu=0, 1$ and $i=1,2.$

By choosing  $s=k,i=k,j=k+1$ in \eqref{df3.14} and using \eqref{df3.4}, we get
\eqa
y^k\,\Gamma_{l+1}^{k,k+1}(y)&& =g^{k,l+1}(y)\Gamma_{l+1}^{k,k+1}(y)+g^{k,l+2}(y)\Gamma_{l+2}^{k,k+1}(y)\nn\\
&&=g^{km}(y)\dfrac{\p g^{k+1,k}(y)}{\p y^m}-\frac{1}{2}g^{k+1,m}(y)\dfrac{\p g^{kk}(y)}{\p y^m}-\dsum_{j=1}^l
g^{kj}(y)\Gamma_j^{k,k+1}(y).\nn \eeqa
Repeat using the degrees and Lemma \ref{lem3.2} and Lemma \ref{alem3.3}, we thus conclude
 $$\mathcal{L}_e(\mathcal{L}_e \Gamma_{l+1}^{k, k+1}(y))=0.$$
 Furthermore, with the help of \eqref{df3.15}, we have $\mathcal{L}_e(\mathcal{L}_e \Gamma_{l+1}^{k+1,\ k}(y))=0.$

Similarly, by choosing $s=k+1,i=k,j=k+1$ in \eqref{df3.14} and using \eqref{df3.4}, \eqref{df3.16} and \eqref{df3.15}, we have
\eqa
y^{k+1}\,\Gamma_{l+2}^{k,k+1}(y)&=&g^{k,l+1}(y)\,\Gamma_{l+1}^{k,k+1}(y)+g^{k+1,l+2}(y)\Gamma_{l+2}^{k,k+1}(y)\nn\\
&=&\frac{1}{2}g^{km}(y)\dfrac{\p g^{k+1,k+1}(y)}{\p y^m}-\dsum_{j=1}^l
g^{k+1,j}(y)\Gamma_j^{k,k+1}(y)\nn
\eeqa
 and
 \beq \Gamma_{l+2}^{k+1,k}(y)=\dfrac{\p g^{k+1,k}}{\p y^{l+2}}-\Gamma_{l+2}^{k,k+1}(y).\nn\eeq
So
$$\mathcal{L}_e(\mathcal{L}_e \Gamma_{l+2}^{k, k+1}(y))=\mathcal{L}_e(\mathcal{L}_e \Gamma_{l+2}^{k+1, k}(y))=0.$$
This completes the proof of the lemma.\end{proof}

\begin{lem}\label{lem3.3} Setting
\beq \eta^{ij}(y)=\mathcal{L}_e g^{ij}(y) \label{dfeta}\eeq
and denoting $\mathcal{R}_{k,k+1}={\mathcal R}_{1}\cup {\mathcal R}_{2}$, then we have

(1) If $\alpha_i$ and $\alpha_j$ belong to different components of $\mathcal{R}_{k,k+1}$, then $\eta^{ij}(y)=0$;

(2) The block $\eta_{(t)}=(\eta^{ij}(y))|_{\alpha_i,\alpha_j\in \mathcal{R}_t}$ of the matrix $(\eta^{ij}(y))$
 corresponding to any branch $\mathcal{R}_t$ has triangular form. The antidiagonal elements of $\eta_{(t)}$ consists
 of the constant numbers $\eta^{ii^*}$ for $\alpha_i\in \mathcal{R}_t$, where $t=1,2$;

 (3) $\eta^{a,l+1}(y)=\varsigma_1\delta_{a,k}, \quad \eta^{a,l+2}(y)=\varsigma_2\delta_{a,k+1}$ for $ a=1,\cdots, l+2$.
\end{lem}
\begin{proof} (1) Let $\alpha_i\in\mathcal{R}_1$ and $\alpha_j\in\mathcal{R}_2$, i.e. $1\leq i<k$ and $k+1<j\leq l$.
As discussed above, if $\beta_{ij}({\bf x})$ as a polynomial in $y_1({\bf x}),\cdots, y_l({\bf x})$
contains a monomial $y_1^{p_1}\cdots y_l^{p_l}$ with $p_k=1$, then
\beq\omega_i+\omega_j=p_1\omega_1+\cdots+p_l\omega_l+\dsum_{s=1}^lq_s\alpha_s \label{df3.10}\eeq
for some nonnegative integers $q_1,\cdots, q_l$.  We multiply $\eqref{df3.10}$ by $\omega_1$ and obtain
$$l+1+k-i-j-\dsum_{s\ne k} p_s (l-s-1)=(l+1) q_1.$$
Since $k-i-j<0$, then we have $q_1=0$ and
$$ l-i+1-(j-k)=\dsum_{s\ne k} p_s (l-s-1) $$
which yields that $p_s=0$ for $s=1,\cdots, i$. So \eqref{df3.10} becomes
 \beq\omega_i+\omega_j=p_{i+1}\omega_{i+1}+\cdots+p_l\omega_l+\dsum_{s=2}^lq_s\alpha_s. \label{df3.11}\eeq
We multiply \eqref{df3.11} by $\alpha_1,\cdots,\alpha_{i}$ and get $q_2=\cdots=q_i=0$, $q_{i+1}=-1$, which
contradicts nonnegativity of $q$'s. So in this case $\dfrac{\p}{\p y^k} g^{ij}(y)=0$.
Similarly, one can show that $\dfrac{\p}{\p y^{k+1}} g^{ij}(y)=0$.  We thus complete the proof of the first statement.

(2) Observe that in any component of $\mathcal{R}_{k,k+1}$, the numbers $d_i$ are distinct and ordered monotonically and
$\deg \eta^{ij}(y)=d_i+d_j-d_k$. We thus conclude that
$\eta^{ij}(y)=0$ when $d_i+d_j=d_k$, and $\eta^{ij}(y)=$ constant when $d_i+d_j=d_k$ which happens if the labels $i$ and
$j$ are dual to each other in the sense of \eqref{df2.14}.

(3) Obviously, the third statement follows from \eqref{df3.4}.
\end{proof}

\begin{prop}\label{prop3.4} If $\varsigma_1\varsigma_2\ne 0$, then the determinant of $(\eta^{ij}(y))$ is a nonzero constant.
\end{prop}

\begin{proof}By using Lemma \ref{lem3.3}, we know that
\beq \det(\eta^{ij}(y)=(-1)^{\frac{k(k+1)+(l-k)(l-k+1)}{2}}\displaystyle\prod_{i=1}^l \eta^{ii^*}.\eeq
It suffices to show that $\eta^{ii^*}$ are nonzero constants for $i=1,\cdots l$.

 For a fixed $1\leq i <k$, with the use of \eqref{df3.5} we obtain
\eqa \eta^{ii*}&=&\eta^{i(k-i)}(y)=\mathcal{L}_eg^{i(k-i)}(y)\nn\\
&=&
 e^{2\pi i (d_{k,k}x_{l+1}+d_{k,k+1}x_{l+2})}\mathcal{L}_e(c_{i(k-i)}y_i({\bf x})y_{k-i}({\bf x})+\beta_{i(k-i)}({\bf x}))
 \nn\\
 &=& e^{2\pi i (d_{k,k}x_{l+1}+d_{k,k+1}x_{l+2})}\mathcal{L}_e \beta_{i(k-i)}({\bf x}) \nn\\
 &=& \dfrac{\p}{\p y_k({\bf x})}\beta_{i(k-i)}({\bf x})+e^{2\pi i \left(\frac{k}{l+1}x_{l+1}-\frac{l-k}{l+1}x_{l+2}\right)}
 \dfrac{\p}{\p y_{k+1}({\bf x})}\beta_{i(k-i)}({\bf x}).\nn
  \eeqa
 Since $\eta^{ii^*}$ is a constant, we thus have
\eqa \frac{1}{\varsigma_1}\eta^{ii*}&=& \dfrac{\p}{\p y_k({\bf x})}\beta_{i(k-i)}({\bf x})=\frac{1}{4\pi^2}
 \dfrac{\p}{\p y_k({\bf x})}\dsum_{p,q=1}^{l}\dfrac{\p y_i({\bf x})}{\p x_p}\dfrac{\p
y_{k-i}({\bf x})}{\p x_q}(\omega_p,\omega_q)\nn
\eeqa
which coincides with the nonzero constant $\eta^{i(k-i)}$ used in the case $\widetilde{W}^{(k)}(A_l)$
(e.g., please see the Corollary 2.3 in \cite{DZ1998}).

Similarly, for a fixed $k+1<i\leq l$, we have
\eqa \frac{1}{\varsigma_2}\eta^{ii*}&=&\eta^{i(l+k+2-i)}(y)=\mathcal{L}_eg^{i(l+k+2-i)}(y)\nn\\
&=&
 e^{2\pi i (d_{k+1,k}x_{l+1}+d_{k+1,k+1}x_{l+2})}\mathcal{L}_e(c_{i(l+k+2-i)}y_i({\bf x})y_{k-i}({\bf x})+\beta_{i(k-i)}({\bf x}))
 \nn\\
 &=& e^{2\pi i (d_{k+1,k}x_{l+1}+d_{k+1,k+1}x_{l+2})}\mathcal{L}_e \beta_{i(l+k+2-i)}({\bf x}) \nn\\
 &=& e^{2\pi i \left(\frac{l-k}{l+1}x_{l+2}-\frac{k}{l+1}x_{l+1}\right)}
 \dfrac{\p}{\p y_{k}({\bf x})}\beta_{i(k-i)}({\bf x})+\dfrac{\p}{\p y_{k+1}({\bf x})}\beta_{i(l+k+2-i)}({\bf x})\nn\\
 &=& \dfrac{\p}{\p y_{k+1}({\bf x})}\beta_{i(l+k+2-i)}({\bf x})=\frac{1}{4\pi^2}
 \dfrac{\p}{\p y_{k+1}({\bf x})}\dsum_{p,q=1}^{l}\dfrac{\p y_i({\bf x})}{\p x_p}\dfrac{\p
y_{l+k+2-i}({\bf x})}{\p x_q}(\omega_p,\omega_q)\nn
  \eeqa
which is exactly the nonzero constant $\eta^{i(l+k+2-i)}$ used in the case $\widetilde{W}^{(k+1)}(A_l)$
(\cite{DZ1998}).

Observe that $\eta^{kk^*}=\varsigma_1$ and $\eta^{(k+1)(k+1)^*}=\varsigma_2$ in Lemma \ref{lem3.3}, we thus complete the proof of this proposition.
\end{proof}

According to Lemma D.1 in \cite{Du1996} (or see Lemma 3.3 in \cite{DSZZ2019})
and using Lemma \ref{lem3.3}, Propsition \ref{prop3.4} and Lemma \ref{lem3.6}, we have

\begin{thm}\label{thm3.7}$g^{ij}(y)$ and $\eta^{ij}(y)$ form a flat pencil of metrics, i.e.,
the metric
\beq g^{ij}(y) + \lambda \eta^{ij}(y)\nn\eeq
is flat for arbitrary $\lambda$ and
the Levi-Civit\`a connection for this metric has the form
$\Gamma_{m}^{ij}(y) + \lambda \gamma_{m}^{ij}(y).$ Here
$\gamma_{m}^{ij}(y):=\mathcal{L}_e\Gamma_{m}^{ij}(y)$
 are the contravariant components of the Levi-Civita connection for the
  metric $(\eta^{ij}(y))$.
 \end{thm}

 Without loss of generality,  in what follows we take
$\varsigma_1=\varsigma_2=1$ unless otherwise stated.

\section{Frobenius manifold structures on the orbit space $\mathcal{M}$}
\label{sec-4}

In this section we want to describe Frobenius manifold structures on the orbit space $\mathcal{M}$ of
${\widetilde W}^{(k,k+1)}(A_l)$ for $1\leq k < l$.

\subsection{The change of coordinates} In order to do this, we firstly make the change
of coordinates
\eqa
&& z^j=y^j, \quad j=1,\cdots, k-1,\, k+2,\cdots, l,  \nn \\
&& z^k=y^k, \quad z^{k+1}=y^{k+1}-y^k,   \label{df4.1}\\
&& z^{l+1}=y^{l+2}, \quad z^{l+2}=y^{l+1}+y^{l+2},\nn \eeqa
such that
\beq e=\dfrac{\p}{\p y^k}+\dfrac{\p}{\p y^{k+1}}=\dfrac{\p}{\p z^k}\,.\eeq

Let us denote
\beq g^{ij}(z)=\dsum_{a,b=1}^{l+2}\dfrac{\p z^i}{\p y^a}\dfrac{\p z^j}{\p
y^b}(dy^a,dy^b)^\thicksim,\quad \eta^{ij}(z)=\dsum_{a,b=1}^{l+2}\dfrac{\p z^i}{\p y^a}\dfrac{\p z^j}{\p
y^b}\eta^{ab}(y) \label{df4.3}\eeq
and $\Gamma_{m}^{ij}(z)$ and $\gamma_{m}^{ij}(z)$ are the contravariant components of the Levi-Civita connection for the
  metric $g^{ij}(z)$ and $\eta^{ij}(z)$.  Under the simple change of coordinates \eqref{df4.1}, it is easy to know that

\begin{enumerate}
  \item $g^{ij}(z)$ and $\Gamma_{m}^{ij}(z)$ are weighted homogeneous polynomials in
 ${z}^1$, $\cdots$, $z^{l+2}$, $e^{z^{l+1}}$, $e^{z^{l+2}-z^{l+1}}$ of the degrees $\deg g^{ij}(z)=\deg z^i+\deg z^j$ and
\beq \deg \Gamma_{m}^{ij}(z)=\deg z^i+\deg z^j-\deg z^m,
\label{df4.4}\eeq
where $\deg z^i=d_i$.
Moreover, $g^{ij}(z)$  and $\Gamma_{m}^{ij}(z)$ are at most linear in $z^k$ and
\eqa  && g^{k ,l+1}(z)=0 ,~\qquad g^{k+1 ,l+1}(z)=z^{k+1}+z^k, \nn \\
&&g^{k,l+2}(z)=z^k,\qquad g^{k+1,l+2}(z)=z^{k+1},  \label{df4.5}\\
&&g^{l+1 ,l+1}(z)=\frac{l-k+1}{l-k},\qquad g^{l+1,l+2}(z)=\frac{1}{l-k}, \nn \\
&&g^{l+2 ,l+2}(z)=\frac{l}{k(l-k)}.\nn
\eeqa

\item $\eta^{ij}(z)=\dfrac{\p g^{ij}(z)}{\p z^k}$ and $\gamma_{m}^{ij}(z)=\dfrac{\p \Gamma_{m}^{ij}(z)}{\p z^k}$. Especially,
\beq \eta^{i,l+1}(z)=\delta_i^{k+1},\quad  \eta^{i,l+2}(z)=\delta_i^{k},\quad i=1,
\cdots, l+2.\label{df4.6}\eeq
So we rename $k^*=l+2$ and $(k+1)^*=l+1$, and also have
$$\eta^{ii^*}(z)=\eta^{ii^*},\quad  i\ne k,k+1,\quad \eta^{kk^*}(z)=\eta^{(k+1)(k+1)^*}(z)=1.$$

\item $g^{ij}(z)$ and $\eta^{ij}(z)$ form a flat pencil of metrics.

\item $\eta_{ij}(z)$ are weighted homogeneous polynomials in
 ${z}^1,\cdots$,~$z^l$, $e^{z^{l+2}-z^{l+1}}$, $e^{z^{l+1}}$, where $(\eta^{ij}(z))$ is the inverse matrix of $(\eta^{ij}(z))$.
     \end{enumerate}


\subsection{Flat coordinates of the metric $\eta^{ij}(z)$}
In this subsection, we want to describe flat coordinates of the metric $\eta^{ij}(z)$.
\begin{lem} For $\nu,s,t=0,1$, we have
 \beq \frac{\p}{\p  z^{k+\nu}}\eta^{k+s,\ k+t}(z)=0.  \label{df4.7}\eeq\end{lem}
 \begin{proof} Observe that
 \eqa \dfrac{\p}{\p  z^{k+1}}\eta^{k,\ k+1}(z)&=&\dfrac{\p}{\p  z^{k+1}} \dfrac{\p }{\p z^k}\dsum_{a,b=1}^{l+2}\dfrac{\p z^k}{\p y^a}\dfrac{\p z^{k+1}}{\p
y^b}g^{ab}(y)\nn\\
 &=&\dfrac{\p}{\p  y^{k+1}}\left(\dfrac{\p}{\p  y^{k}}+\dfrac{\p}{\p  y^{k+1}}\right)\left(g^{k,k+1}(y)-g^{kk}(y)\right)=0 \nn\eeqa
 which follows from \eqref{ddf3.12}. The other cases are similar.
 \end{proof}

 \begin{lem}\label{EZZlem4.6} For $\nu=0,1$, we have
 \beq \gamma_{j}^{i,\ l+1+\nu}(z)=0,\quad i,j=1,\cdots, l+2.\label{df4.8} \eeq\end{lem}
 \begin{proof} With the use of \eqref{df4.4} and
 $\gamma_{j}^{i,\ l+1+\nu}(z)=\dfrac{\p}{\p z^k}\Gamma_{j}^{i,\ l+1+\nu}(z)$, we obtain
 \beq \deg \gamma_{j}^{i,\ l+1+\nu}(z)=d_i-d_j-d_k<0, \quad i.e.,\quad \gamma_{j}^{i,\ l+1+\nu}(z)=0\label{df4.9} \eeq
 except the cases $\gamma_{l+1+\mu}^{k+\sigma,\ l+1+\nu}(z)$ for $\nu,\mu, \sigma=0,1.$
So it suffices to show that
\beq
 \gamma_{l+1+\mu}^{k+\sigma,\ l+1+\nu}(z)=0, \quad \nu,\mu, \sigma=0,1. \label{df4.10}\eeq

Since $\gamma_{m}^{ij}(z)$ are the contravariant components of the Levi-Civita connection for the
metric $g^{ij}(z)$, then
  \beq
2\eta^{sm}(z)\gamma_m^{ij}(z)=\eta^{im}(z)\dfrac{\p \eta^{js}(z)}{\p
z^m}+\eta^{sm}(z)\dfrac{\p \eta^{ji}(z)}{\p z^m} -\eta^{jm}(z)\dfrac{\p
\eta^{is}(z)}{\p z^m}.\label{df4.11}\eeq
By choosing $i=k+\sigma$ and $j=l+1+\nu$ in \eqref{df4.11}, it follows from \eqref{df4.9} that
\eqa &&2\eta^{s,\ l+1}(z)\gamma_{l+1}^{k+\sigma,\ l+1+\nu}(z)+2\eta^{s,\ l+2}(z)\gamma_{l+2}^{k+\sigma,\ l+1+\nu}(z)
=\eta^{k+\sigma,\ m}(z)\dfrac{\p \eta^{l+1+\nu,\ s}(z)}{\p z^m}
\nn\\&&+~~\eta^{sm}(z)\dfrac{\p \eta^{l+1+\nu,\ k+\sigma}(z)}{\p z^m} -\eta^{l+1+\nu,\ m}(z)\dfrac{\p
\eta^{k+\sigma,\ s}(z)}{\p z^m}.\label{df4.12}\eeqa
Taking $s=k$ and $s=k+1$ respectively in \eqref{df4.12} and with the help of \eqref{df4.6} and \eqref{df4.7},
we get the desired identities \eqref{df4.10}. \end{proof}

\begin{thm}\label{thm-4.6}There exist flat coordinates of the metric $(\eta^{ij}(z))$ in the form
\beq \begin{array}{l} t^\alpha=z^\alpha+h^\alpha({z}^1, \cdots,\widehat{z^\alpha},\cdots, z^l, e^{z^{l+1}}, e^{z^{l+2}-z^{l+1}}),
 \\
t^{l+1}=z^{l+1},\quad t^{l+2}=z^{l+2}, \quad \alpha=1,\cdots, l,
\end{array} \label{EZZ4.22}\eeq
where $h^\alpha$ are weighted homogeneous polynomials in ${z}^1, \cdots,z^{\alpha-1},z^{\alpha+1},\cdots, z^l$, $e^{z^{l+1}}$,
$e^{z^{l+2}-z^{l+1}}$ of degree ${d}_\alpha$ defined in
\eqref{df2.11}.
\end{thm}

\begin{proof}Local existence of the coordinates $t^1,\cdots,t^{l+2}$ follows from
flatness of the metric $(\eta^{ij}(z))$. The flat coordinates $t=t(z)$ are to be found
from the following system
\beq
 \frac{\pal^2 t}{\pal z^i\pal z^j}+\eta_{is}(z)\gamma_j^{sm}(z)
   \frac{\pal t}{\pal z^m}=0, \quad i,j=1,\cdots, l+2. \label{ZZ3.19}
\eeq
The system \eqref{ZZ3.19} can be written as linear differential equations
\beq
 \dfrac{\pal w_i}{\pal z^j}-\gamma_{ij}^{m}(z)w_m=0,\quad
w_i=\dfrac{\p t}{\p z^i}.\nn\eeq
This is an overdetermined holonomic system. So the space of solutions has dimension $l+2$.
Observe that  those coefficients in \eqref{ZZ3.19} are
weighted homogeneous polynomials in
 ${z}^1$, $\cdots$, $z^l$, $e^{z^{l+1}}$, $e^{z^{l+2}-z^{l+1}}$.
From \eqref{df4.8}, it follows that
\beq t^{l+1}=z^{l+1},\quad t^{l+2}=z^{l+2}\nn\eeq
are two solutions of \eqref{ZZ3.19}. We choose remaining solutions
\beq t^\alpha=z^\alpha+h^\alpha({z}^1, \cdots,\widehat{z^\alpha},\cdots, z^l, e^{z^{l+1}}, e^{z^{l+2}-z^{l+1}}),
\quad \alpha=1,\cdots,l\nn\eeq
in such a way that
\beq \dfrac{\p t^\alpha}{\p z^\beta}(0,\cdots,0,0,0)=\delta_\beta^\alpha,\quad \alpha,
\beta=1,\cdots, l.\label{ZZZ1}\eeq
These solutions $t^\alpha$ are power series in ${z}^1, \cdots, z^l$, $e^{z^{l+1}}$,
$e^{z^{l+2}-z^{l+1}}$. The system \eqref{ZZ3.19} is invariant with respect
to the transformation
\beq z^\alpha \mapsto  c_0^{d_\alpha} z^\alpha, z^{l+1}  \mapsto z^{l+1}+\frac{1}{l-k} \log(c_0),
z^{l+2}  \mapsto z^{l+2}+ \frac{l}{k(l-k)}\log(c_0)\nn \eeq
for any positive constant $c_0$. This yields that  $t^\alpha$ are weighted
homogeneous in ${z}^1, \cdots, z^l$, $e^{z^{l+1}}$,
$e^{z^{l+2}-z^{l+1}}$ of the same degree ${d}_\alpha>0$. Thus $h^\alpha$ are
 weighted homogeneous polynomials in  ${z}^1, \cdots,z^{\alpha-1},z^{\alpha+1},\cdots, z^l$, $e^{z^{l+1}}$,
$e^{z^{l+2}-z^{l+1}}$ of degree  ${d}_\alpha$. \end{proof}

\begin{cor}\label{cor3.13}In the flat coordinates $t^1,\cdots, t^{l+2}$, the entries
of the metric $(\eta^{ij}(t))$ have the form
\beq
\eta^{ij}:=\eta^{ij}(t)=\left\{\begin{array}{cll}
\eta^{ii^*}(y),\quad & j=i^*,\\
0, \quad& j\ne i^*,\\
\end{array}\right.\nn\eeq
for $i,j=1,\cdots,l+2$. Especially,
\beq \eta^{i,l+1}(t)=\delta_i^{k+1},\quad  \eta^{i,l+2}(t)=\delta_i^{k},\quad i=1,
\cdots, l+2.\nn\eeq
The entries of the matrix $(g^{ij}(t))$ and the
Christoffel symbols $\Gamma^{ij}_m(t)$ are weighted homogeneous
polynomials of ${t}^1, \cdots, t^l$, $e^{t^{l+1}}$,
$e^{t^{l+2}-t^{l+1}}$  of
degrees $d_i+d_j$ and $d_i+d_j-d_m$ respectively.
In particular,
\eqa
&& g^{\al,\,l+2}(t)=d_\al t^\al,~~ 1\leq \al\leq l,\quad g^{l+1 ,l+1}(t)=\frac{l-k+1}{l-k},\nn \\
&& g^{l+1,l+2}(t)=\frac{1}{l-k},\quad g^{l+2 ,l+2}(t)=\frac{l}{k(l-k)},\label{zh8}
\eeqa
and
\beq \Gamma_{j}^{l+2,\,i}(t)= d_j\,\delta_{j}^i,\quad  1\leq i,j \leq l+2 \label{EZZ4.30} \eeq
and
\beq  g^{k+1,l+1}(t)=t^{k}+t^{k+1}+g_0(t)\label{ZZZ99}\eeq
for certain weighted homogenous polynomial $g_0(t)$ in  $t^1,\cdots$,${t}^{k-1}$,
${t}^{k+2}$, $\cdots$,
 $t^l$, $e^{t^{l+1}}$, $e^{t^{l+2}-t^{l+1}}$ of degree 1.
\end{cor}
\begin{proof} In the flat coordinates $t^1,\cdots, t^{l+2}$,  using \eqref{ZZZ1} and Theorem \ref{thm-4.6} we have
$$\dfrac{\p}{\p t^k}=\dfrac{\p}{\p z^k},\quad \eta^{ij}(t)=\dfrac{\p g^{ij}(t)}{\p t^k}.$$ Thus the first statement
of this corollary follows from the fact that the linear part of $t^\al$ is $z^\al$.

By definition, we easily obtain \eqref{zh8}. The identity \eqref{ZZZ99} follows from \eqref{df4.5} and \eqref{EZZ4.22}. It remains
to prove \eqref{EZZ4.30}. Notice that $t^{l+2}=2\pi i(x_{l+1}+x_{l+2})$ and
\eqa \dsum_{j=1}^{l+2}\Gamma_{j}^{l+2,i}(t)dt^j
&& =\dsum_{p,q,r=1}^{l+2}\frac{\p t^{l+2}}{\p x_p}
   \frac{\p^2 t^i}{\p x_q \p x_r}(dx_p,dx_q)^\thicksim dx_r \nn \\
&&=\dsum_{r=1}^{l+2}\dsum_{p,q=l+1}^{l+2}\frac{\p t^{l+2}}{\p x_p}
   \frac{\p^2 t^i}{\p x_q \p x_r}(dx_p,dx_q)^\thicksim dx_r\nn\\
&& =-4\pi^2\dsum_{p=l+1}^{l+2}
    [d_{i,k}(dx_p,dx_{l+1})^\thicksim+d_{i,k+1}(dx_p,dx_{l+2})^\thicksim]dt^i\nn\\
&&=(\frac{d_{i,k}}{k}+  \frac{d_{i,k+1}}{l-k})dt^i=d_i dt^i.\nn
\eeqa
This completes the proof of the corollary.
\end{proof}

\subsection{Frobenius manifold structures on the orbit space $\mathcal{M}$}
Now we are ready to describe the Frobenius manifold structures on the orbit space of
the extended affine Weyl group ${\widetilde W}^{(k,k+1)}(A_l)$.

 \begin{thm}\label{thm4.7}
 For any fixed integer $1\le k<l$, there exists a unique Frobenius manifold structure
 of charge $d=1$ on the orbit
space $\mathcal{M}$ of $\widetilde{W}^{(k,k+1)}(A_l)$ such that
the potential $F(t)=\widehat{F}(t)+\frac{1}{2}(t^{k+1})^2\log(t^{k+1})$, where $\widehat{F}(t)$ is
a weighted homogeneous polynomial in $t^1,t^2$, $\cdots, t^{l+2}$, $e^{t^{l+1}}$, $e^{t^{l+2}-t^{l+1}}$, satisfying
\begin{enumerate}
\item the unity vector field $e$ coincides with $\dfrac{\p}{\p t^k}$;
 \item the Euler vector field has the form
\begin{equation}E=\dsum_{\alpha=1}^{l}d_\alpha t^\alpha \dfrac{\p}{\p t^\alpha}
+\dfrac{1}{l-k}\dfrac{\p}{\p t^{l+1}}+\dfrac{l}{k(l-k)}\dfrac{\p}{\p t^{l+2}}~,\label{AM00}
\end{equation}
where $d_1,\dots,d_{l}$ are defined in \eqref{df2.11};
\item the invariant flat metric and the intersection form of the Frobenius manifold structure coincide respectively with the
metrics $\eta^{ij}$ and $g^{ij}(t)$ on $\mathcal{M}$.
\end{enumerate}
\end{thm}

\begin{proof}The idea of the proof is similar to that of \cite{DZ1998}, i.e., using the theory of
flat pencils of metrics (\cite{Du1996}).

 Let $\Gamma^{\alpha\beta}_{\gamma}(t)$\
be the coefficients
of the Levi-Civita connection for the metric $\ (\ ,\ )^{\sptilde}$\
in the coordinates $\ t^1,\dots,t^{l+2}$.\  According to Proposition D.1 of \cite{Du1996}
one can represent these functions as
\begin{equation}
\Gamma^{\alpha\beta}_{\gamma}(t)=\eta^{\alpha\ve}\pal_{\ve}\pal_{\gamma}
f^{\beta}(t)\label{deq2.28}
\end{equation}
for some functions $\ f^{\beta}(t)$.\ From the weighted homogeneity of
$\ \Gamma^{\alpha\beta}_{\gamma}(t)$\ and Corollary \ref{cor3.13}, one has
$$
\partial_\alpha \partial_\gamma \left( {\mathcal{L}}_E f^\beta - {(1+
d_\beta)} f^\beta \right) =0
$$
for any $\alpha, ~\beta$. So
\begin{equation}
{\mathcal{L}}_E f^\beta (t) = {(d_\beta +1)}f^\beta(t)
+A^\beta_\sigma t^\sigma +B^\beta\label{deq2.29}
\end{equation}
for some constants $A^\beta_\sigma$, $B^\beta$.
Doing a transformation
$$f^\beta(t)\mapsto \widetilde f^\beta (t)=f^\beta(t) +R^\beta_\lambda
t^\lambda+ Q^\beta
$$
all the coefficients $A^\beta_\sigma$, $B^\beta$ in \eqref{deq2.29} can be killed
except $A^{l+1+\tau}_{k+\nu}$, for $\nu,\,\tau=0,\,1$. Indeed, after the transformation,
\begin{equation*}
\begin{aligned}
{\mathcal{L}}_E \widetilde f^\beta (t)& = {(d_\beta+1)}
\widetilde f^\beta (t)+
\sum_{\gamma =1}^l \left[ R^\beta_\gamma ({d_\gamma -1
-d_\beta })+A^\beta_\gamma \right]t^\gamma \\
&+\left[ A_{l+1}^\beta - ({1 + d_\beta})
R^\beta _{l+1}\right] t^{l+1}+\left[ A_{l+2}^\beta - ({1 + d_\beta})
R^\beta _{l+2}\right] t^{l+2}\\
&+\dfrac{1}{l-k}\,R^\beta_{l+1}+\dfrac{l}{k(l-k)}\,R^\beta_{l+2}+ B^\beta -{(1+d_\beta)}Q^\beta
\end{aligned}
\end{equation*}
The function $\widetilde f^\beta (t)$ does still satisfy \eqref{deq2.28}. Choosing
\begin{equation*}
\begin{aligned}
&R^\beta_{l+1} ={1 \over 1 +d_\beta} A^\beta_{l+1}\, , \quad R^\beta_{l+2} ={1 \over 1 +d_\beta} A^\beta_{l+2}\, ,
\\ &
Q^\beta ={1 \over 1 +d_\beta} \left[ B^\beta + \dfrac{1}{l-k}R^\beta_{l+1}++ \dfrac{l}{k(l-k)}R^\beta_{l+2}
\right]\, ,\end{aligned}
\end{equation*}
one kills the constant term in the r.h.s. of \eqref{deq2.29} and the term linear in
$t^{l+1}$ and $t^{l+2}$. In order to kill other linear terms, putting
$$R^\beta_\gamma = {1 \over d_\beta +d_{\gamma^*}}A^\beta_\gamma,
$$
where $\gamma^*$ is the index dual to $\gamma$ in the sense of
duality defined in the above subsection. We can do this unless $d_\beta = d_{\gamma^* }=0$.
The last equation holds only for $$\beta =l+1+\tau,  \quad \gamma = k+\nu, \quad \nu,\, \tau =0,\, 1.$$
So all linear terms can be killed
except $A^{l+1+\tau}_{k+\nu}$ for $\nu,\,\tau=0,\,1$ in \eqref{deq2.29}.
Thus for $\beta \neq l+1, \, l+2$ the polynomials $f^\beta (t)$ can be
assumed to be homogeneous of the degree $d_\beta +1$.

Next we want to show that {\it for $1\leq \beta \leq l$ the functions $f^\beta(t)$
are polynomials in $t^1, \dots, t^l,$ $e^{t^{l+1}}$,
 $e^{t^{l+2}-t^{l+1}}$.}
We already know
that this is true for the Christoffel coefficients
$\Gamma_\gamma^{\alpha\beta}(t)$.
Denoting
$$
\eta_{\alpha\epsilon} \Gamma^{\epsilon\beta}_{l+i} (t) = \sum_{m=0}^{M_i} \sum_{n=0}^{N_i}
C_{\alpha, m,n}^{\beta,i} e^{{m\,t^{l+1}+n\, ({t^{l+2}-t^{l+1}})}}\equiv
\partial_\alpha \partial_{l+i} f^\beta (t),~~i=1,\,2,
$$
where the coefficients $C_{\alpha, m,n}^{\beta,i}$ are polynomials in $t^1,
\dots, t^l$ and $M_i,\, N_i$ are certain positive integers. From the compatibility condition
$$
\partial_{l+i}\left( \partial _\alpha \partial_{l+j} f^\beta(t) \right)
= \partial_\alpha \left( \partial_{l+i}\p_{l+j} f^\beta(t) \right),\quad i,\,j=1,2,
$$
one obtains
$$
\partial_\alpha C_{l+i,\, 0,\,0}^{\beta,j} =0, \quad \alpha = 1, \dots, l, \quad i, j=1,2.
$$
So $C_{l+i, \,0,\,0}^{\beta,j}$ are constants. But $\partial_{l+i}\p_{l+j} f^\beta(t)$
must be
a weighted homogeneous polynomial in $t^1, \dots, t^l$, $e^{t^{l+1}}$, $e^{{t^{l+2}-t^{l+1}}}$
 of the positive degree $1+d_\beta$. Thus
 \beq C_{l+i, \,0,\,0}^{\beta,j}=0,\quad i,~j=1,~2.\nn\eeq
 and
\beq\partial_\alpha \partial_{l+i} f^\beta (t)=\sum_{m=1}^{M_i}
\sum_{n=1}^{N_i} C_{\alpha, m,n}^{\beta,i} e^{{m\,t^{l+1}+n\, ({t^{l+2}-t^{l+1}})}},
~~ i=1,2.\label{ZZZ4.13}\eeq
From the compatibility condition
\eqa
\partial_{l+i}\left( \partial _{l+j} \partial_{l+j} f^\beta(t) \right)
=\partial_{l+j} \left( \partial_{l+i}\p_{l+j} f^\beta(t) \right) =\partial_{l+j} \left( \partial_{l+j}\p_{l+i} f^\beta(t) \right),
\nn
\eeqa
for $i,\,j=1,2$, one gets $M:=M_1=M_2$ and $N:=N_1=N_2$ and
\beq  C_{l+1,m,n}^{\beta,\, 1}=\frac{m-n}{m}C_{l+2,m,n}^{\beta,\,1}
=\frac{m-n}{m}C_{l+1,m,n}^{\beta,\,2}=\frac{(m-n)^2}{m^2}C_{l+2,m,n}^{\beta,\,2}.\label{ZZZ4.14}\eeq
It follows from \eqref{ZZZ4.13} and \eqref{ZZZ4.14} that
$$
f^\beta(t) = \sum_{m=1}^M\sum_{n=1}^N {1 \over m^2} C_{l+2, m,n}^{\beta,2} e^{{m\,t^{l+1}+n\, ({t^{l+2}-t^{l+1}})}}
+ t^{l+1} D_1^\beta +t^{l+2} D_2^\beta + D_0^\beta
$$
for some new polynomials $D_i^\beta = D_i^\beta (t^1, \dots, t^l),~ i=0,1,2$. Since
$\partial_\alpha \partial_\gamma f^\beta(t)$ must not contain terms
linear in $t^{l+1}$ and $t^{l+2}$, these polynomials $D_i^\beta$ are at most linear
in $t^1, \dots, t^l$. Using the homogeneity of $f^\beta(t)$, so $D_i^\beta =0,~~i=1,\,2$ and
\beq f^\beta(t) = \sum_{m=1}^M\sum_{n=1}^N {1 \over m^2} C_{l+2, m,n}^{\beta,2}
e^{{m\,t^{l+1}+n\, ({t^{l+2}-t^{l+1}})}}+ D_0^\beta,\quad \beta=1,\cdots,l\,.\nn\eeq

The coefficients $ \Gamma^{\al\beta}_{\gamma}(t)$\ must also
satisfy the conditions
\begin{equation}
g^{\al\sigma}\Gamma^{\beta\gamma}_{\sg}=g^{\beta\sg}\Gamma^{\al
\gamma}_{\sg}.\label{deq2.30}
\end{equation}
For $\ \al=l+2$,\ it follows from \eqref{deq2.28}, \eqref{deq2.30}
and \eqref{EZZ4.30} that
\beq
\mathcal{L}_E(\eta^{\beta\ve}\pal_{\ve}f^{\gamma})=d_{\gamma}
g^{\beta\gamma}.
\label{ZZZ4.15}\eeq
Notice that the right side has no summation w.r.t the index $\gamma$.
Because of $\ \deg f^{\gamma}=d_{\gamma}+1$,\ then
$$\ \deg(\eta^{\beta\ve}\pal_{\ve} f^{\gamma})=d_{\beta}+d_{\gamma},\quad \gamma\ne l+1,~~l+2.$$
Thus,
\begin{equation}
(d_{\gamma}+d_{\beta})\eta^{\beta\ve}\pal_{\ve}
f^{\gamma}=d_{\gamma} g^{\beta\gamma},\quad \gamma\ne l+1,~~l+2.\label{deq2.31}
\end{equation}
Putting
$$
F^{\gamma}=\frac{1}{d_{\gamma}} f^{\gamma},\quad \gamma\ne l+1,~l+2
$$
and using \eqref{deq2.31}, one has
\begin{equation}
\eta^{\beta\ve}\pal_{\ve} F^{\gamma}=\eta^{\gamma\ve}\pal_{\ve}F^{\beta},
\quad 1\le \gamma,\beta\le l.\label{deq2.32}
\end{equation}
From \eqref{deq2.32} it follows that a function $F=F(t)$ exists such that
\begin{equation}
F^\gamma = \eta^{\gamma\epsilon} \partial_\epsilon F,
\quad 1\leq\gamma\leq l.\label{deq2.33}
\end{equation}
The dependence of $F$ on $t^k$ is not determined from \eqref{deq2.33}. However,
 putting $\beta = l+2$ in \eqref{deq2.31}, one obtains
\begin{equation}
\partial_k F^\gamma = t^\gamma , ~~1\leq \gamma \leq l.\label{deq2.34}
\end{equation}
By using \eqref{deq2.33} and  \eqref{deq2.34}, one gets
\begin{equation*}
\begin{aligned}
&\partial_{l+2} \left( \partial_k F\right) = t^k,\quad
\partial_{l+1} \left( \partial_k F\right) =t^{k+1}=\sum_{\al=1}^l \eta_{l+1,\,\al} t^{\al},\\
&\partial_\gamma \left( \partial_k F \right) =
\sum_{\al=1}^l \eta_{\gamma\al} t^{\al},\quad \gamma\neq k, k+1, l+1, l+2.
\end{aligned}
\end{equation*}
Notice that $(l+1)^*=k+1$, hence
$$\partial_k F = t^k t^{l+2}+ {1\over 2} \sum_{\alpha,\ \beta \neq k,\ l+2}
\eta_{\alpha\beta} t^\alpha t^\beta + \mu(t^k)
$$
for some function $\mu(t^k)$. Shifting $F(t) \mapsto F(t)+\int \mu(t^k) \, dt^k$
one can kill this function, and the equations in \eqref{deq2.33} still hold
true due to $\ \eta^{ik}=\delta_{i,l+2}$.\
So one has the representation
\begin{equation}
F(t) = {1\over 2} \left(t^k\right)^2 t^{l+2} +{1\over 2} t^k
\sum_{\alpha, \beta \neq k,l+2} \eta_{\alpha\beta} t^\alpha t^\beta + G(t)
\label{deq2.35}
\end{equation}
with some $G(t)$ independent on $t^k$.

From the definition \eqref{deq2.33} of $F$ and the weighted
homogeneity of $f^\gamma,\
\gamma\ne l+1, \, l+2$, it follows that
\beq
{\mathcal{L}}_E F(t) = 2 F(t) + \rho(t^k,t^{k+1}) \label{ZZZ4.21}
\eeq
for certain unknown function $\rho(t^k, t^{k+1})$. By using  \eqref{deq2.35} and the duality condition of the degrees,
one has
\eqa
{\mathcal{L}}_E F(t)&&=\left( t^k \right)^2 t^{l+2} + t^k \sum_{\alpha, \beta \neq k,l+2}
\eta_{\alpha\beta} t^\alpha t^\beta \nn \\
 && \quad+\, {\mathcal{L}}_E G(t)+\frac{l}{2k(l-k)} \left( t^k\right)^2+\frac{1}{l-k}t^k t^{k+1}. \label{ZZZ4.22}
\eeqa
From \eqref{ZZZ4.21} and \eqref{ZZZ4.22}, one gets
$$
\mathcal{L}_E G(t) = 2 G(t)+\rho(t^k,t^{k+1})-\frac{l}{2k(l-k)} \left( t^k\right)^2-\frac{1}{l-k}t^k t^{k+1} .
$$
But $\mathcal{L}_E G(t)$ does not depend on $t^k$, thus there exists an unknown function $\varphi(t^{k+1})$ with
$\deg \varphi(t^{k+1})=2$,
 which does not depend on $t^k$, such that
$$\varphi(t^{k+1})=s\rho(t^k,t^{k+1})-\frac{l}{2k(l-k)} \left( t^k\right)^2-\frac{1}{l-k}t^k t^{k+1}$$
and
$$
\mathcal{L}_E G(t) = 2 G(t)+\varphi(t^{k+1})+c_0,
$$
for a constant $c_0$. Killing the constant by a shift, it follows
that
\beq \mathcal{L}_E G(t) = 2 G(t)+\varphi(t^{k+1}),\label{ZZZ4.23}\eeq
and
 $G(t)=G( t^1, \dots,t^{k-1},t^{k+1},\dots, t^{l+1}, e^{t^{l+1}},e^{t^{l+2}-t^{l+1}})$ is a
weighted homogeneous function of the degree $2$. The above
conditions determine this function uniquely. By integrating polynomials, it yields that $G$ is
a polynomial in $t^1, \dots, t^{k-1}$,\,  $t^{k+2}$, $\dots, t^{l+1}$, $e^{t^{l+1}}$,\
$e^{t^{l+2}-t^{l+1}}$. But because of the existence of $\varphi(t^{k+1})$,
$G$ may have some terms which are not polynomials in $t^{k+1}$.

Substituting $F(t)$ into \eqref{ZZZ4.15}, one gets
\beq g^{\alpha\beta}
 =\mathcal{L}_E F^{\alpha\beta}, \quad F^{\alpha\beta}=\eta^{\alpha\epsilon}\eta^{\beta\delta}\p_\epsilon\p_\delta F.\label{deq3.33} \eeq
For $\alpha =l+1, \beta = l+2$,  the
equation \eqref{deq3.33} reads
\beq \mathcal{L}_E {\partial^2 F(t)\over \partial t^{k}
\partial t^{k+1}}=\dfrac{1}{l-k}. \label{ZZZ4.25}
\eeq
For $\alpha =l+2, \beta = l+2$ the
equation, \eqref{deq3.33} reads
\beq \mathcal{L}_E {\partial^2 F(t)\over \partial t^{k}
\partial t^{k}}=\dfrac{l}{k(l-k)}. \label{ZZZ4.26}
\eeq
For $\alpha = \beta = l+1$, the
equation \eqref{deq3.33} reads
\beq \mathcal{L}_E {\partial^2 F(t)\over \partial t^{k+1}
\partial t^{k+1}}=\dfrac{l-k+1}{l-k}.
\label{ZZZ4.27}\eeq
By using \eqref{ZZZ4.22} and \eqref{ZZZ4.23}, we get
\beq \mathcal{L}_E {\partial^2 G(t)\over \partial t^{k+1}
\partial t^{k+1}}=\dfrac{l-k+1}{l-k} \nn\eeq
and \beq \p^2_{k+1}\varphi(t^{k+1})= \p^2_{k+1} \mathcal{L}_E
G(t)-2  \p^2_{k+1} G(t) =\dfrac{l-k+1}{l-k}.\label{ZZZ4.28}\eeq
 Notice that $\deg\varphi=2$ and integrating \eqref{ZZZ4.28} to obtain
\beq \varphi(t^{k+1})=\dfrac{l-k+1}{2(l-k)}
(t^{k+1})^2.\label{ZZZ4.29} \eeq By using \eqref{deq2.33} and
\eqref{deq2.35}, one can write \beq G(t)=H_0(t)+\psi(t^{k+1})+a_0
t^{l+1}(t^{k+1})^2,\quad \deg \psi(t^{k+1})=2. \label{ZZZ4.30}\eeq
 Here $a_0$ is a constant and $H_0(t)$ is a weighted homogeneous polynomial in
 $t^1, \cdots, t^{k-1}$, $t^{k+1}$,\, $\dots, t^{l+1}$, $e^{t^{l+1}}$,\
$e^{t^{l+2}-t^{l+1}}$ of degree $2$ with $\mathcal{L}_E H_0(t)=2H_0(t)$. In particular, $H_0(t)$ is at
most linear in $t^{k+1}$. By substituting \eqref{ZZZ4.29} and
\eqref{ZZZ4.30} into \eqref{ZZZ4.23}, one has
 \beq t^{k+1}\dfrac{\p \psi(t^{k+1})}{\p
{t^{k+1}}}=2\psi(t^{k+1})+\left(\dfrac{k+1}{2k}-\frac{a_0}{k}\right)(t^{k+1})^2.
\label{ZZZ4.31}\eeq The solutions of \eqref{ZZZ4.31} are \beq
\psi(t^{k+1})=c_1\ (t^{k+1})^2+
\left(\dfrac{l-k+1}{2(l-k)}-\frac{a_0}{l-k}\right)\,
(t^{k+1})^2\log(t^{k+1}),\label{ZZZ4.32} \eeq where $c_1$ is
 an integral constant.

For $\alpha =k+1, \beta = l+1$,  the equation \eqref{deq3.33} reads
\beq g^{k+1,l+1}(t)=\mathcal{L}_E {\partial^2 F(t)\over \partial t^{k+1}
\partial t^{l+1}}=t^k+{2a_0}t^{k+1}+\mathcal{L}_E {\partial^2 H_0(t)\over \partial t^{k+1}
\partial t^{l+1}}.  \label{ZZZ100}
\eeq
By comparing \eqref{ZZZ99} with \eqref{ZZZ100}, one gets
$$ a_0=\frac{1}{2},\quad \mathcal{L}_E {\partial^2 H_0(t)\over \partial t^{k+1}
\partial t^{l+1}}=g_0(t). $$
 Thus,
$$G(t)=H(t)+\dfrac{1}{2}\, (t^{k+1})^2\log(t^{k+1}),$$
where $H(t)=H_0(t)+c_1\ (t^{k+1})^2+\dfrac{1}{2}(t^{k+1})^2t^{l+1}$.

 In a word, we show that the existence of a unique weighted homogeneous
polynomial
$$H(t):=H(t^1,\dots,t^{k-1},t^{k+1},\dots,t^{l+1},e^{t^{l+1}}, e^{t^{l+2}-t^{l+1}})$$
of degree $2$ such that the function \begin{equation} F = {1\over
2} \left(t^k\right)^2 t^{l+2} +{1\over 2} t^k \sum_{\alpha, \beta
\neq k,l+2} \eta_{\alpha\beta} t^\alpha t^\beta
+\dfrac{1}{2}(t^{k+1})^2\log(t^{k+1})+H(t)
\label{zz2}
\end{equation}
 satisfies the equations
 \beq \label{zz3} g^{\alpha\beta}
 =\mathcal{L}_E F^{\alpha\beta},\quad \Gamma^{\alpha\beta}_\gamma=d_\beta\,
c^{\alpha\beta}_\gamma,\quad \alpha,\beta,\gamma=1,\dots,l+2, \eeq
where $c^{\alpha\beta}_{\gamma}=\dfrac{\pal F^{\alpha\beta}}{\pal t^\gamma}$.

Obviously, the function $F$ satisfies the equations \beq
\frac{\pal^3 F}{\pal t^k \pal t^i \pal t^j}=\eta_{ij},\quad
i,j=1,\dots, l+2 \eeq and the quasi-homogeneity condition \beq
{\mathcal L}_E F=2 F+\frac{l}{2k(l-k)} \left(
t^k\right)^2+\frac{1}{l-k}t^k t^{k+1}+\dfrac{l-k+1}{2(l-k)}
(t^{k+1})^2. \eeq From the properties of a flat pencil of metrics
\cite{Du1996}, it follows that $F$ also satisfies associativity
equations of WDVV \beq\label{zz7} c_{m}^{ij}\,c_{q}^{m
p}=c_{m}^{ip}\,c_{q}^{m j} \eeq for any set of fixed indices
$i,j,p,q$. Now the theorem follows from above properties of the
function $F$ and the simple identity $\mathcal{L}_E e=-e$.  This
completes the proof of the theorem. \end{proof}

\subsection{Examples}
We end this section by giving some examples to illustrate the above
construction.  For the brevity,
instead of $t^1,\dots,t^{l+2}$ we will
denote the flat coordinates of the metric $\eta^{ij}$ by
$t_1,\dots, t_{l+2}$, and also denote $\p_i=\frac{\p}{\p
t_i}$ in this subsection.

\begin{ex} \label{ex5.1} Let $\widetilde{W}$ be the extended affine Weyl group
$\widetilde{W}^{(1,2)}(A_2)$, then
\beq d_{1,1}=\frac{2}{3},\quad  d_{1,2}=d_{2,1}=\frac{1}{3}, \quad d_{2,2}=\frac{2}{3}\nn  \eeq
 and
\begin{eqnarray*}
&&y^1={e^{\frac{2\pi\,i}{3}\,(2x_3+x_4)}}(e^{2\pi\,i x_1}+e^{2\pi\,i x_2}+e^{-2\pi\,i (x_1+x_2)}), \\
&&y^2={e^{\frac{2\pi\,i}{3}\,(x_3+2 x_4)}}(e^{-2\pi\,i x_1}+e^{-2\pi\,i x_2}+e^{2\pi\,i (x_1+x_2)}), \\
&&y^3=2\,i\pi\,x_{{3}}, \quad y^4=2\,i\pi\,x_{{4}}.
\end{eqnarray*}
The metric $(~,~)^{\sptilde}$ has the form
\begin{equation*}
((dx_a,dx_b)^{\sptilde})=\frac{1}{4\pi^2}\left( \begin
{array}{cccc} ~~\frac{2}{3}&~~\frac{1}{3}\,&~~0&~~0\\ ~~\frac{1}{3}& ~~\frac{2}{3}&~~0&~~0\\
~~0&~~0&~~-2&~~1\\ ~~0&~~0&~~1&-2\end {array}
 \right).
 \end{equation*}
 We thus have
$$ (g^{ij}(y))=\left(\dsum_{a,b=1}^{4}\dfrac{\p y^i}{\p x_a}\dfrac{\p
y^j}{\p x_b}(dx_a,dx_b)^{\sptilde}\right)=\left( \begin {array}{cccc} 2\,y^{{2}}{e^{y^{{3}}}}&3\,{e^{y^{{3}}+y^{{4}}}}&y^{{1}}&0\\
\noalign{\medskip}3\,{e^{y^{{3}}+y^{{4}}}}&2\,y^{{1}}{e^{y
^{{4}}}}&0&y^{{2}}\\
\noalign{\medskip}y^{{1}}&0&2&-1
\\
\noalign{\medskip}0&y^{{2}}&-1&2\end {array} \right)$$
and
$$(\eta^{ij}(y))=\left(\mathcal{L}_eg^{ij}(y)\right)\left( \begin {array}{cccc} 2\,{e^{y^{{3}}}}&0&1&0\\
\noalign{\medskip}0&2\,{e^{y^{{4}}}}&0&1\\
\noalign{\medskip}1&0&0&0
\\
\noalign{\medskip}0&1&0&0\end {array} \right),\quad e=\frac{\p}{\p y^1}+\frac{\p}{\p y^2}.$$
The flat coordinates for the metric $(\eta^{ij}(y))$ are
\eqa && t_1=y^1-e^{y^3},\quad t_2=-y^1+y^2+e^{y^3}-e^{y^4},\quad t_3=y^4,\quad  t_4=y^3+y^{{4}}.\nn\eeqa
The potential has the expression
\eqa
F=\frac{1}{2}\,{t_{{1}}}^{2}t_{{4}}+t_1t_2t_3+\frac{1}{2}\,t_{{2}}^2t_{{3}}
+\,e^{t_4}-t_2e^{t_3}+t_2e^{t_4-t_3}+ {\bf \frac{1}{2}t_2^2\log t_2}\nn
\eeqa
and the unit vector field is
 $$e=\p_1$$
and the Euler vector field is given by
$$E=t_1{\p_1}+t_2{\p_2}+{\p_3}+2{\p_4}.$$
\end{ex}

\begin{ex}Let $\widetilde{W}$ be the extended affine Weyl group $\widetilde{W}^{(1,2)}(A_3)$,
then
\beq d_{1,1}=\frac{3}{4}, \quad d_{1,2}=\frac{1}{2}, \quad d_{1,3}=\frac{1}{4}, \quad d_{2,2}=1,
\quad d_{2,3}=\frac{1}{2},\quad d_{3,3}=\frac{3}{4}\nn\eeq
 and
 \begin{eqnarray*}
&&y^1={e^{\frac{\pi\,i}{2}\,(3x_4+2x_5)}}\dsum_{j=1}^4{\xi_a}, \quad
y^2={e^{{\pi\,i}\,(x_4+2x_5)}}\dsum_{1\leq a<b\leq 4} {\xi_a\xi_b}, \\
&&y^3={e^{\frac{\pi\,i}{2}\,(x_4+2x_5)}}\dsum_{1\leq a<b<c\leq 4} {\xi_a\xi_b\xi_c}, \quad
y^4=2\,i\pi\,x_{{4}},\quad  y^5=2\,i\pi\,x_{{5}},
\end{eqnarray*}
where $\xi_j=e^{2\pi\,i (x_j-x_{j-1})}$ for $j=1,2,3$ and $x_0=0$, $\xi_4=e^{-2\pi\,i x_3}$.
The metric $(~,~)^{\sptilde}$ has the form
\begin{equation*}
((dx_i,dx_j)^{\sptilde})=\frac{1}{4\pi^2}\left( \begin
{array}{ccccc} \frac{3}{4}&~~\frac{1}{2}&~~\frac{1}{4}\,&~~0&~~0\\
\frac{1}{2}&~~1& ~~\frac{1}{2}&~~0&~~0\\
\frac{1}{4}&~~\frac{1}{2}&~~\frac{3}{4}&~~0&~~0\\
 0&~~0&~~0&~~-2&~~\frac{1}{2}\\ 0&~~0&~~0&~~\frac{1}{2}&-\frac{3}{2}\end {array}
 \right).
 \end{equation*}
 We thus have
 $$(g^{ij}(y))=\left( \begin {array}{ccccc} 2\,y^{{2}}{e^{y^{{4}}}}&3\,y^{{3}}{e^{y^{{4}}+y^{{5}}}}&4\,{e^{y^{{4}}+y^{{5}}}}&y^{{1}}&0\\
 \noalign{\medskip}
3\,y^{{3}}{e^{y^{{4}}+y^{{5}}}}&2\,y^{{1}}y^{{3}}{e^{y^{{5}}}}+4\,{e^{
2\,y^{{5}}+y^{{4}}}}&3\,y^{{1}}{e^{y^{{5}}}}&0&y^{{2}}
\\
\noalign{\medskip}4\,{e^{y^{{4}}+y^{{5}}}}&3\,y^{{1}}{e^{y^{{5}}}}&-
\frac{1}{2}\,({y^{{3}}})^{2}+2\,y^{{2}}&0&\frac{1}{2}\,y^{{3}}\\\noalign{\medskip}y^{{1}
}&0&0&2&-1\\\noalign{\medskip}0&y^{{2}}&\frac{1}{2}\,y^{{3}}&-1&\frac{3}{2}
\end {array} \right)
$$
and
 $$(\eta^{ij}(y))=\left(\mathcal{L}_eg^{ij}(y)\right)=\left( \begin {array}{ccccc} 2\,{e^{y^{{4}}}}&0&0&1&0\\
 \noalign{\medskip}
0&2\,y^{{3}}{e^{y^{{5}}}}&3\,{e^{y^{{5}}}}&0&1
\\
\noalign{\medskip}0&3\,{e^{y^{{5}}}}&2&0&0\\
\noalign{\medskip}1&0&0&0&0\\
\noalign{\medskip}0&1&0&0&0
\end {array} \right),\quad e=\frac{\p}{\p y^1}+\frac{\p}{\p y^2}.
$$
To write down the flat coordinates, we first introduce the following variables
\eqa z^1=y^1,~ z^2=y^{{2}}-y^{{1}},~ z^3=y^3, ~ z^4=y^5,~z^5=y^4+y^5,\nn \eeqa
and obtain
$$(\eta^{ij}(z))=\left( \begin {array}{ccccc} 2\,{e^{z^{{5}}-z^{{4}}}}&-2\,{e^{z^{{5}}-z^{{4}}}}&0&0&1\\
\noalign{\medskip}-2\,{e^{z^{{5}}-z^{{4}}}}&2\,{e^{z^{{5}}-z^{{4}}}}+2\,z^{{3}}{e^{z^{{4}}}}&3\,{e^{z^{{4}}}}&1&0
\\
\noalign{\medskip}0&3\,{e^{z^{{4}}}}&2&0&0\\
\noalign{\medskip}0&1&0&0&0\\
\noalign{\medskip}1&0&0&0&0\end {array} \right),
$$
then the flat coordinates for the metric $(\eta^{ij}(z)$ are given by
\eqa
&&t_1=z^{{1}}-e^{z^5-z^4},\quad t_2=z^{{2}}-z^3e^{z^4}+e^{z^5-z^4}+e^{2z^4},\nn\\
&&t_3=z^3-e^{z^4}, \quad t_4=z^4,\quad t_5=z^5.\nn
\eeqa
The potential has the expression
\eqa F&=& \frac{1}{2}\,{t_{{1}}}^{2}t_{{5}}+t_{{1}}t_{{2}}t_{{4}}+\frac{1}{2}\,{t_{{2}}}^{2}t_{{
4}}+\frac{1}{4}\,{t_{{3}}}^{2}t_{{2}}+\frac{1}{4}\,{t_{{3}}}^{2}t_{{1}}-{\frac {1}{96}
}\,{t_{{3}}}^{4}\nn\\&+&t_{{3}}{e^{t_{{5}}}}-t_{{2}}t_{{3}}{e^{t_{{4}}}}+t_{{2}}{e^{
t_{{5}}-t_{{4}}}}+\frac{1}{2}\,t_{{2}}{e^{2\,t_{{4}}}}+{\bf \frac{1}{2}\,{t_{{2}}}^{
2}\log  \left( t_{{2}} \right)}
\nn
\eeqa
and the unit vector field is
$$e=\p_1$$
and the Euler vector field is given by
$$E=t_1{\p_1 }+t_2{\p_2}+\frac{1}{2}t_3{\p_3}
+\frac{1}{2}{\p_4}+\frac{3}{2}{\p_5}.$$

\end{ex}

\begin{ex}Let $\widetilde{W}$ be the extended affine Weyl group $\widetilde{W}^{(2,3)}(A_3)$,
then
\beq d_{1,1}=\frac{3}{4}, \quad d_{1,2}=\frac{1}{2}, \quad d_{1,3}=\frac{1}{4}, \quad d_{2,2}=1,
\quad d_{2,3}=\frac{1}{2},\quad d_{3,3}=\frac{3}{4}\nn\eeq
 and
 \begin{eqnarray*}
&&y^1={e^{\frac{\pi\,i}{2}\,(2x_4+x_5)}}\dsum_{j=1}^4{\xi_a}, \quad
y^2={e^{{\pi\,i}\,(2x_4+x_5)}}\dsum_{1\leq a<b\leq 4} {\xi_a\xi_b}, \\
&&y^3={e^{\frac{\pi\,i}{2}\,(2x_4+3x_5)}}\dsum_{1\leq a<b<c\leq 4} {\xi_a\xi_b\xi_c}, \quad
y^4=2\,i\pi\,x_{{4}},\quad  y^5=2\,i\pi\,x_{{5}},
\end{eqnarray*}
where $\xi_j=e^{2\pi\,i (x_j-x_{j-1})}$ for $j=1,2,3$ and $x_0=0$, $\xi_4=e^{-2\pi\,i x_3}$.
The metric $(~,~)^{\sptilde}$ has the form
\begin{equation*}
((dx_i,dx_j)^{\sptilde})=\frac{1}{4\pi^2}\left( \begin
{array}{ccccc} \frac{3}{4}&~~\frac{1}{2}&~~\frac{1}{4}\,&~~0&~~0\\
\frac{1}{2}&~~1& ~~\frac{1}{2}&~~0&~~0\\
\frac{1}{4}&~~\frac{1}{2}&~~\frac{3}{4}&~~0&~~0\\
 0&~~0&~~0&~~-\frac{3}{2}&~~1\\ 0&~~0&~~0&~~1&-2\end {array}
 \right).
 \end{equation*}
We thus have
$$ (g^{ij}(y))=\left( \begin {array}{ccccc} - \frac{1}{2}\,({y^{{1}}})^{2}+2\,y^{{2}}&
3\,y^{{3}}{e^{y^{{4}}}}&4\,{e^{y^{{4}}+y^{{5}}}}&\frac{1}{2}\,y^{{1}}&0
\\
\noalign{\medskip}3\,y^{{3}}{e^{y^{{4}}}}&2\,y^{{1}}y^{
{3}}{e^{y^{{4}}}}+4\,{e^{y^{{5}}+2\,y^{{4}}}}&3\,y^{{1}}{e^{y^{{4}}+y^{{5}}}}&y^{{2
}}&0\\\noalign{\medskip}4\,{e^{y^{{4}}+y^{{5}}}}&3\,y^{{1}}{e^{y^{{4}}+y^{{5}
}}}&2\,y^{{2}}{e^{y^{{5}}}}&0&y^{{3}}\\\noalign{\medskip}\frac{1}{2}\,y
^{{1}}&y^{{2}}&0&\frac{3}{2}&-1\\\noalign{\medskip}0&0&y^{{3}}&-1&2
\end {array} \right)
$$
and
$$(\eta^{ij}(y))=\left(\mathcal{L}_eg^{ij}(y)\right)=\left( \begin {array}{ccccc} 2&3\,{e^{y^{{4}}}}&0&0&0\\\noalign{\medskip}3\,{e^{y^{{4}}}}&2\,y^{{1}}{e^{y^{{4}}}}&0&1&0
\\\noalign{\medskip}0&0&2\,{e^{y^{{5}}}}&0&1\\\noalign{\medskip}0&1&0&0
&0\\\noalign{\medskip}0&0&1&0&0\end {array} \right),\quad e=\frac{\p}{\p y^2}+\frac{\p}{\p y^3}.$$
To write down the flat coordinates, we first introduce the following variables
\eqa z^1=y^1,~ z^2=y^{{2}},~ z^3=y^3-y^2,~z^4=y^5,~z^5=y^4+y^5,\nn \eeqa
and obtain
$$(\eta^{ij}(z))=\left( \begin {array}{ccccc} 2&3\,{e^{z^{{5}}-z^{{4}}}}&-3\,{e^{z^{{5}}-z^{{4}}}}&0&0\\\noalign{\medskip}3\,{e^{z^{{5}}-z^{{4}}}}&2\,{e^{z^
{{5}}-z^{{4}}}}z^{{1}}&-2\,z^{{1}}{e^{z^{{5}}-z^{{4}}}}&0&1
\\\noalign{\medskip}-3\,{e^{z^{{5}}-z^{{4}}}}&-2\,z^{{1}}{e^{z^{{5}}-z^{{4}}}
}&2\,z^{{1}}{e^{z^{{5}}-z^{{4}}}}+2\,{e^{z^{{4}}}}&1&0
\\\noalign{\medskip}0&0&1&0&0\\\noalign{\medskip}0&1&0&0&0\end {array}
 \right),
$$
then the flat coordinates for the metric $(\eta^{ij}(z)$ are given by
\eqa
&&t_1=z^{{1}}-e^{z^5-z^4},\quad t_2=z^{{2}}-z^1e^{z^5-z^4}+e^{2z^5-2z^4},\nn\\
&&t_3=z^3-e^{z^4}+z^1e^{z^5-z^4}-e^{2(z^5-z^4)}, \quad t_4=z^4,\quad t_5=z^5.\nn
\eeqa
The potential has the expression
\eqa F &=& \frac{1}{2}\,{t_{{2}}}^{2}t_{{5}}+t_{{2}}t_{{3}}t_{{4}}+\frac{1}{2}\,{t_{{3}}}^{2}t_{{
4}}+\frac{1}{4}\,{t_{{1}}}^{2}t_{{2}}-{\frac {1}{96}}\,{t_{{1}}}^{4}\nn\\
&+&t_{{1}}{e^{t_{{5}}}}-t_{{3}}{e^{t_{{4}}}}+t_{{1}}t_{{3}}{e^{t_{{5}}-t_{{4}}}}-\frac{1}{2}\,{e^{2
\,t_{{5}}-2\,t_{{4}}}}t_{{3}}+{\bf \frac{1}{2}\,{t_{{3}}}^{2}\log  \left(
t_{{3}} \right)}
\nn
\eeqa
and the unit vector field is
$$e=\p_1$$
and the Euler vector field is given by
$$E=\frac{1}{2}t_1{\p_1 }+t_2{\p_2}+t_3{\p_3}
+{\p_4}+\frac{3}{2}{\p_5}.$$

\end{ex}

\section{The group ${\widetilde{W}^{(k,k+1)}}(A_l)$ and the Hurwitz space $\mathbb{M}_{k,l-k+1,1}$}\label{sec-6}
In this section we want to show that the space $\mathbb{M}_{k,l-k+1,1}$  as a Frobenius
manifold is isomorphic to the orbit space $\mathcal{M}$ of the extended affine
Weyl group $\widetilde{W}^{(k,k+1)}(A_l)$, where
$\mathcal{M}={\mathcal M}^{(k,k+1)}\setminus\{\ty_{l+1}=0\}\cup\{\ty_{l+2}=0\}$ for $1\leq k<l$.

 Let $\mathbb{M}_{k,l-k+1,1}$ be the space of a particular class of LG superpotentials consisting of trigonometric-Laurent
 series of one variable with tri-degree $(k+1,l-k,1)$,
these being functions of the form
\beq \lambda(\varphi)=(e^{{\bf i}\varphi}-a_{l+2})^{-1}(e^{{\bf i}(k+1)\varphi}+a_1e^{{\bf i}k\varphi}+\cdots+a_{l+1} e^{{\bf i}(k-l)\varphi}), \quad
 a_{l+1}a_{l+2}\ne 0,\label{df5.1} \eeq
where $a_j \in \mathbb{C}$ for $j=1,\cdots, l+2$. For brevity, we denote $m=l-k$ in this section.
 According to \cite{Du1996}, the space $\mathbb{M}_{k,m+1,1}$ is a simple Hurwitz space and
carries a natural structure of Frobenius manifold. The invariant inner product
$\eta$ and the intersection form $g$ of two vectors $\p'$, $\p''$
tangent to $M_{k,m+1,1}$ at a point $\lambda(\varphi)$ can be defined by
the following formulae
\beq \widetilde{\eta}(\p',\p'')=(-1)^{k+1}\dsum_{|\lambda|<\infty}
\res_{d\lambda=0}
\dfrac{\p'(\lambda(\varphi)d\varphi)\p''(\lambda(\varphi)d\varphi)}{d\lambda(\varphi)},
\label{AM2}\eeq
and
\beq \widetilde{g}(\p',\p'')=-\dsum_{|\lambda|<\infty}
\res_{d\lambda=0}
\dfrac{\p'(\log\lambda(\varphi)d\varphi)\p''(\log\lambda(\varphi)d\varphi)}{d\log\lambda(\varphi)}.
\label{AM3}\eeq
In these formulae, the derivatives
$\p'(\lambda(\varphi)d\varphi)$ $etc.$ are to be calculated keeping $\varphi$ fixed.
The formulae \eqref{AM2} and \eqref{AM3} uniquely determine
multiplication of tangent vectors on $\mathbb{M}_{k,m+1,1}$ assuming that the
Euler vector field $E$  has the form
\beq
E=\dsum_{j=1}^{l+1}\frac{j}{k}a_j\frac{\p}{\p
a_j}+\frac{1}{k}a_{l+2}\frac{\p}{\p
a_{l+2}}.\label{AM4}\eeq
For any tangent vectors $\p'$, $\p''$ and $\p'''$ to $M_{k,m+1,1}$ , one has \beq
c(\p',\p'',\p''')=\widetilde{\eta}(\p'\cdot \p'',\p''')=-\dsum_{|\lambda|<\infty} \res_{d\lambda=0}
\dfrac{\p'(\lambda(\varphi)d\varphi)\p''(\lambda(\varphi)d\varphi)\p''(\lambda(\varphi)d\varphi)}{d\lambda(\varphi)d\varphi}.
\label{AM6}\eeq
The canonical coordinates $u_1,\cdots, u_{l+2}$ for this multiplication are
the critical values of $\lambda(\varphi)$ and
\beq \p_{u_\alpha}\cdot \p_{u_\beta}=\delta_{\alpha\beta}\p_{u_\alpha}, \quad
\mbox{where}\quad \p_{u_\alpha}=\frac{\p}{\p {u_\alpha}}.\label{AM7}\eeq

We start with factorizing
\beq \lambda(\varphi)=(e^{i\varphi}-e^{i\varphi_{l+2}})^{-1}
e^{-im\varphi}\prod_{b=1}^{l+1}(e^{i\varphi}-e^{i\varphi_b}).\label{AM8}\eeq



 \begin{thm}\label{Main2}
Let $\mathfrak{f}: \mathcal{M} \to \mathbb{M}_{k,m+1,1}$ be induced by the map
\beq (x_1,\cdots, x_{l+2})\mapsto (\varphi_1,\cdots,\varphi_{l+2})\label{AM9}\eeq
with
\beq\begin{array}{l}
\varphi_1=2\pi(\rho+x_1),\quad \varphi_j=2\pi(\rho+x_j-x_{j+1}), ~~j=2,\cdots, l,\\
\varphi_{l+1}=2\pi(\rho-x_{l}),\quad \varphi_{l+2}=2\pi x_{l+1}, \end{array}
\label{AM10}
\eeq
where $\rho=\frac{m+1}{l+1}x_{l+1}+\frac{m}{l+1}x_{l+2}$.
Then $\mathfrak{f}$ is an $m$-fold covering map,  which is also a local isomorphism
between the Frobenius manifolds
$\mathcal{M}$ and $\mathbb{M}_{k,m+1,1}$.
\end{thm}

\begin{proof} With the use of \eqref{df5.1} and \eqref{AM8}, one obtains
$$a_j=(-1)^j\sigma_j(e^{i\varphi_1},\cdots,e^{i\varphi_{l+1}}), \quad a_{l+2}=e^{i\varphi_{l+2}},\quad j=1,\cdots,l+1.$$
From the formulae \eqref{df1.2}, \eqref{df2.9} and \eqref{AM10}, it follows
\beq
a_j=\left\{
\begin{array}{ll}
(-1)^j y^j,& j=1,\cdots, k,\\
 (-1)^{j}y^{j}e^{(j-k)y^{l+1}+(j-k-1)y^{l+2}},& j=k+1,\cdots,l, \\
 (-1)^{l+1}e^{(m+1)y^{l+1}+my^{l+2}},& j=l+1, \\
e^{y^{l+1}}, & j=l+2.
\end{array}\right.\label{df5.10}
\eeq
Then the map $\mathfrak{f}: (\tilde{y}_1,\cdots, \tilde{y}_{l+2})\mapsto
(a_1,\cdots, a_{l+2})$ is given by
\eqa  \label{AM11}
&& a_j=(-1)^j \tilde{y}_{j},\quad j=1,\cdots, k,\nn\\
&& a_{k+s}=(-1)^{k+s} \tilde{y}_{k+s} \tilde{y}_{l+1}^{s} \tilde{y}_{l+2}^{s-1},\quad s=1,\cdots,l-k,\\
&&a_{l+1}=(-1)^{l+1}\tilde{y}_{l+1}^{m+1}\tilde{y}_{l+2}^m,\quad a_{l+2}=\tilde{y}_{l+1}\nn
\eeqa
and the Jacobian of $\mathfrak{f}$ is proportional to $\tilde{y}_{l+1}^{m+1}\tilde{y}_{l+2}^{m-1}$. So $\mathfrak{f}$
is an $m$-fold covering map.

Now let us proceed to prove that $\mathfrak{f}$ is a local isomorphism
between the two Frobenius manifolds.  By using \eqref{AM11}, it is easy to check that the Euler vector fields
\eqref{AM4} and \eqref{AM00} coincide. So it suffices to prove that the
intersection form \eqref{AM3} coincides with the intersection form
of the orbit space, and the metric \eqref{AM2} coincides with the metric \eqref{dfeta}.

Let us denote the roots of $\lambda'(\varphi)$ by $\psi_\gamma,~ \gamma=1,\dots,
l+2$, then
\beq\lambda'(\varphi)=(k+1)i(e^{i\varphi}-e^{i\varphi_{l+2}})^{-2}
e^{-im\varphi}\prod_{\gamma=1}^{l+2}(e^{i\varphi}-e^{i\psi_\gamma}). \label{AM12}\eeq
We define   $u_\alpha=\lambda(\psi_\alpha)$, $\alpha=1,\dots,
l+2$, then
\beq
  \p_{u_\alpha}\lambda(\varphi)|_{\varphi=\psi_\beta}=\delta_{\alpha\beta}.
\label{AM13}\eeq
Using \eqref{AM8}, \eqref{AM13} and the Lagrange interpolation formula one obtains
\beq
 \p_{u_\alpha}\lambda(\varphi)=\frac{ie^{i\psi_\alpha}\lambda'(\varphi)}{(e^{i\varphi}
 -e^{i\psi_\alpha})\lambda''(\psi_\alpha)}.
\label{AM14}\eeq
It follows from \eqref{AM8} and \eqref{AM14} that
\beq
\frac{ie^{i\psi_\alpha}\lambda'(\varphi)}{(e^{i\varphi}
 -e^{i\psi_\alpha})\lambda''(\psi_\alpha)}
 =\lambda(\varphi) \frac{i\p_{u_\alpha}\varphi_{l+2} e^{i\varphi_{l+2}}}{e^{i\varphi}-e^{i\varphi_{l+2}}}
 -\sum_{s=1}^{l+1} \frac{i\p_{u_\alpha}\varphi_{s} e^{i\varphi_{s}}\lambda(\varphi)}{e^{i\varphi}-e^{i\varphi_s}}.
\label{AM15}\eeq
Putting $\varphi=\varphi_b$ in \eqref{AM15} for $b=1,\cdots,l+1$, then
\beq \p_{u_{\alpha}} \varphi_b=- \frac{ie^{i\psi_\alpha}}{(e^{i\varphi_b}
 -e^{i\psi_\alpha})\lambda''(\psi_\alpha)},\quad b=1,\cdots,l+1. \label{AM16}\eeq
Let us  rewrite \eqref{AM15} as
\beq
\frac{ie^{i\psi_\alpha}\lambda'(\varphi)}{(e^{i\varphi}
 -e^{i\psi_\alpha})\lambda''(\psi_\alpha)\lambda(\varphi)^2}
 = \frac{i\p_{u_\alpha}\varphi_{l+2} e^{i\varphi_{l+2}}}{(e^{i\varphi}-e^{i\varphi_{l+2}})\lambda(\varphi)}
 -\sum_{b=1}^{l+1} \frac{i\p_{u_\alpha}\varphi_{b} e^{i\varphi_{b}}}{(e^{i\varphi}-e^{i\varphi_b})\lambda(\varphi)}.
\label{AM17}\eeq
Also, putting $\varphi=\varphi_{l+2}$ in \eqref{AM17} one gets
\beq \p_{u_{\alpha}} \varphi_{l+2}=- \frac{ie^{i\psi_\alpha}}{(e^{i\varphi_{l+2}}
 -e^{i\psi_\alpha})\lambda''(\psi_\alpha)}. \label{AM18}\eeq
We denote
\beq\begin{array}{l}
 \varpi_1=x_1,\quad \varpi_j=x_j-x_{j-1},\quad
  \varpi_{l+1}=x_{l+1}, \quad \varpi_{l+2}=x_{l+2}
\end{array}\label{AM19}\eeq
for $j=2,\cdots,l$.
With the help of \eqref{AM10}, \eqref{AM16} and \eqref{AM18}, one has
\eqa
\p_{u_{\alpha}} \varpi_{\beta}&& =\frac{e^{i\psi_\alpha}}{2\pi i(e^{i\varphi_{\beta}}
-e^{i\psi_\alpha})\lambda''(\psi_\alpha)}\nn\\
&&
\quad -\frac{m+1}{l+1}\p_{u_{\alpha}} \varpi_{l+1}
 -\frac{m}{l+1}\p_{u_{\alpha}} \varpi_{l+2},\nn \\
 \p_{u_{\alpha}} \varpi_{l+1}&&=\frac{e^{i\psi_\alpha}}{2\pi i(e^{i\varphi_{l+2}}
 -e^{i\psi_\alpha})\lambda''(\psi_\alpha)}, \label{AM21}\\
\p_{u_{\alpha}} \varpi_{l+2}&& =\sum_{b=1}^{l+1}\frac{e^{i\psi_\alpha}}{2\pi m i(e^{i\varphi_{b}}
 -e^{i\psi_\alpha})\lambda''(\psi_\alpha)}
 \nn \\&&\qquad -\frac{(m+1)e^{i\psi_\alpha}}{2\pi m i(e^{i\varphi_{l+2}}
 -e^{i\psi_\alpha})\lambda''(\psi_\alpha)}. \nn
\eeqa
From \eqref{AM2},  \eqref{AM3} and \eqref{AM13} one gets
\beq  \widetilde{\eta}_{\alpha\beta}=
\widetilde{\eta}(\p_{u_\alpha}\p_{u_\beta})=(-1)^{k+1}\frac{\delta_{\alpha\beta}}{\lambda''(\psi_\alpha)}
\label{AM23}\eeq
and
\beq \widetilde{g}_{\alpha\beta}=
\widetilde{g}(\p_{u_\alpha}\p_{u_\beta})=-\frac{\delta_{\alpha\beta}}{u_\alpha\lambda''(\psi_\alpha)}.
\label{AM24}\eeq
It follows from \eqref{df5.10} that  the vector field $e=\frac{\p}{\p y^k}+\frac{\p}{\p y^{k+1}}$ in the coordinates
$a_1,\cdots, a_{l+2}$ coincides with
$e=(-1)^k(\frac{\p}{\p a_{k}}-a_{l+2}\frac{\p}{\p a_{k+1}}).$
We shift
\beq a_k\longmapsto a_k+c,\quad a_{k+1}\longmapsto a_{k+1}-ca_{l+2}\nn\eeq
which produces the corresponding shift
\beq u_i\longmapsto u_i+c, \quad i=1,\cdots,l+2 \nn\eeq
of the critical values. This shift does not change the critical points $\psi_\alpha$
neither the values of the second derivative $\lambda''(\psi_\alpha)$. So
\beq \mathcal{L}_e\widetilde{g}^{\alpha\beta}= \mathcal{L}_e(-u_\alpha \lambda''(\psi_\alpha)\delta_{\alpha\beta})
=(-1)^{k+1}\lambda''(\psi_\alpha)\delta_{\alpha\beta}=\widetilde{\eta}^{\alpha\beta}.\nn\eeq

Finally we want to show that the bilinear form \eqref{AM3} in the coordinates $x_1$, $\cdots$, $x_{l+2}$
coincides with the form defined in \eqref{df3.10} and \eqref{df3.02}. We shall use the following identity
\eqa && \sum_{\alpha=1}^{l+2}
\frac{u_\alpha e^{2i\psi_\alpha}}{(e^{i\varphi_{a}}
 -e^{i\psi_\alpha})(e^{i\varphi_{b}}-e^{i\psi_\alpha})\lambda''(\psi_\alpha)}
 \nn\\
 &=&\sum_{\alpha=1}^{l+2} \res_{\varphi=\psi_\alpha} \frac{\lambda(\varphi)u_\alpha e^{2i\varphi}}
{(e^{i\varphi}-e^{i\varphi_{a}}
)(e^{i\varphi}-e^{i\varphi_{b}})\lambda'(\varphi)} \nn \\
&=& \left\{
\begin{array}{l}
\delta_{ab}-\frac{1}{k}, \quad 1\leq a,b\leq l+1,\\
-\frac{1}{k}, \quad 1\leq a \leq l+1,~~ b=l+2,\\
-1-\frac{1}{k}, \quad a=b=l+2,
\end{array}\right.\label{AM26}
\eeqa
which follows from the explicit form of $\lambda(\varphi)$. With the use of \eqref{AM21},
\eqref{AM24} and \eqref{AM26}, then
\eqa
&&(d\varpi_{l+1},d\varpi_{l+1})
=\sum_{\alpha=1}^{l+2} \frac{1}{\widetilde{g}_{\alpha\alpha}(u)}
\p_{u_{\alpha}}\varpi_{l+1}\p_{u_{\alpha}}\varpi_{l+1} \nn \\
&&\qquad =\frac{1}{4\pi^2}\sum_{\alpha=1}^{l+2}\frac{u_\alpha e^{2i\psi_\alpha}}{(e^{i\varphi_{l+2}}
 -e^{i\psi_\alpha})^2\lambda''(\psi_\alpha)}=-\frac{1}{4\pi^2}\frac{k+1}{k} \nn
\eeqa
and
\eqa
&&(d\varpi_{l+1},d\varpi_{l+2})
=\sum_{\alpha=1}^{l+2} \frac{1}{\widetilde{g}_{\alpha\alpha}(u)}
\p_{u_{\alpha}}\varpi_{l+1}\p_{u_{\alpha}}\varpi_{l+2} \nn \\
&&\qquad =\frac{1}{4\pi^2m}\sum_{\alpha=1}^{l+2}\sum_{b=1}^{l+1}\frac{u_\alpha e^{2i\psi_\alpha}}{(e^{i\varphi_{b}}
 -e^{i\psi_\alpha})(e^{i\varphi_{l+2}}
 -e^{i\psi_\alpha})\lambda''(\psi_\alpha)}\nn \\
 && \qquad \qquad-\frac{m+1}{4\pi^2m}\sum_{\alpha=1}^{l+2}\frac{u_\alpha e^{2i\psi_\alpha}}{(e^{i\varphi_{l+2}}
 -e^{i\psi_\alpha})^2\lambda''(\psi_\alpha)}\nn \\
 &&\qquad = \frac{1}{4\pi^2m}\sum_{b=1}^{l+1}({-\frac{1}{k}})+ \frac{(m+1)(k+1)}{4\pi^2mk}=\frac{1}{4\pi^2}\nn
\eeqa
and
\eqa &&(d\varpi_{l+2},d\varpi_{l+2})=\sum_{\alpha=1}^{l+2} \frac{1}{\widetilde{g}_{\alpha\alpha}(u)}
\p_{u_{\alpha}}\varpi_{l+2}\p_{u_{\alpha}}\varpi_{l+2} \nn \\
&&\qquad =\frac{1}{4\pi^2m^2}\sum_{a,b=1}^{l+1}({\delta_{ab}-\frac{1}{k}})
-\frac{2(m+1)}{4\pi^2m^2}\sum_{a=1}^{l+1}({-\frac{1}{k}})+\frac{(m+1)^2}{4\pi^2m^2}(-1-\frac{1}{k})\nn\\
&&\qquad=-\frac{1}{4\pi^2} \frac{m+1}{m}
\nn\eeqa
and for $b=1,\cdots, l$,
\beq (d\varpi_{b},d\varpi_{l+1})= (d\varpi_{b},d\varpi_{l+2})=0
\nn\eeq
and for $1\leq s,b\leq l$,
\eqa
(d\varpi_{a},d\varpi_{b})
=\sum_{\alpha=1}^{l+2} \frac{1}{\widetilde{g}_{\alpha\alpha}(u)}
\p_{u_{\alpha}}\varpi_{a}\p_{u_{\alpha}}\varpi_{b}=\frac{1}{4\pi^2}(\delta_{ab}-\frac{1}{l+1}). \nn\eeqa
By using \eqref{AM19} and the above explicit forms, it is easy to verify that the
intersection form \eqref{AM3} coincides with \eqref{df3.01} and \eqref{df3.02}. This completes the proof of the theorem.
\end{proof}


\section{Conclusions}\label{sec-7}

We have presented a new class of extended affine Weyl groups $\widetilde{W}^{(k,k+1)}(A_l)$ for $1\leq k <l$.
On the orbit spaces $\mathcal{M}^{(k,k+1)}(A_l)\setminus\{\tilde{y}_{l+1}=0\}\cup\{\tilde{y}_{l+2}=0\}$, we have
shown the existence of Frobenius manifold structures and constructed LG superpotentials for these Frobenius manifold
structures. Besides these, there are still some open problems deserved further study.
\begin{itemize}
  \item Is it possible to obtain an explicit realization of the integrable hierarchies associated with the  Frobenius
manifolds on the orbit space of $\widetilde{W}^{(k,k+1)}(A_l)$?  Perhaps this problem is related to the works in \cite{MST2014} or \cite{BCRR} about
rational reductions of the 2D-Toda hierarchy, or in \cite{CHL2018,L2019} about the finite Toda lattice of
CKP type when $k=1$ and $l=2$.

  \item How about the almost dual structure of the resulting Frobenius manifold structures? (\cite{Du2004,SS2016})
  \item Whether the resulting Frobenius manifolds could be regarded as Frobenius submanifolds in
Strachan'sense (\cite{S2001}) of certain infinite-dimensional Frobenius manifolds (\cite{CDM2011, WD2012, WZ2014}) or not?
\end{itemize}
~\\
\noindent {\bf Acknowledgments.}The author is grateful to Professor Youjin Zhang
for bringing me the attention of this project and helpful discussions. This work is partially
supported by NSFC (No.11671371, No.11871446) and Wu Wen-Tsun Key Laboratory of Mathematics, USTC, CAS.


\end{document}